\documentclass[11pt,t]{article}
\usepackage{blindtext}
\usepackage{geometry}
\usepackage{enumerate}
\usepackage{latexsym, amssymb, amsmath}

\usepackage[utf8]{inputenc} 
\usepackage[T1]{fontenc}    
\usepackage[hidelinks]{hyperref}       
\usepackage{url}            
\usepackage{booktabs}       
\usepackage{amsfonts}       
\usepackage{amsmath}
\usepackage{amsthm}
\usepackage{amssymb}
\usepackage{nicefrac}       
\usepackage{microtype}      
\usepackage{cleveref}       
\usepackage{lipsum}         
\usepackage{graphicx}
\usepackage{doi}
\usepackage{color}
\usepackage{tikz-cd}
\usepackage{lmodern}
\usepackage{mathtools}
\usepackage[thinc]{esdiff}
\usepackage{todonotes}
\usepackage{amscd} 
\usepackage{float}
\usepackage{tikz}
\usepackage{pgfplots}
\pgfplotsset{compat=1.15}
\usepackage{mathrsfs}
\usepackage{tkz-euclide}
\usetikzlibrary{arrows}

\newtheorem{theorem}{Theorem}[section]
\newtheorem{lemma}[theorem]{Lemma}
\newtheorem{corollary}[theorem]{Corollary}
\newtheorem{proposition}[theorem]{Proposition}

\theoremstyle{definition}
\newtheorem{definition}[theorem]{Definition}
\newtheorem{remark}{Remark}[section]
\newtheorem{example}{Example}[theorem]

\date{}

\usepackage[backend=bibtex,style=numeric,natbib=true,backref,maxbibnames=99]{biblatex} 

\begin{document}
	
	\title{Extensions of Steiner triple systems\footnote{AMS MSC 05B07, 20N05, 51E10\newline Keywords: Steiner triple systems, Steiner loops, Veblen points, Schreier extensions.}}
	
	\maketitle	\noindent	
	{{Giovanni Falcone}
		\\	\footnotesize{Dipartimento di Matematica e Informatica}\\
		\footnotesize{Universit\`a degli Studi di Palermo, Via Archirafi 34, 90123 Palermo, Italy}\\
		\footnotesize{giovanni.falcone@unipa.it}}
	
	\bigskip\noindent
	{{Agota Figula}
		\\	\footnotesize{Institute of Mathematics}\\
		\footnotesize{University of Debrecen, P.O. Box 400,	H-4002 Debrecen, Hungary}\\
		\footnotesize{figula@science.unideb.hu}}

	\bigskip\noindent
	{{Mario Galici}
		\\
		\footnotesize{Dipartimento di Matematica e Informatica}\\
		\footnotesize{Universit\`a degli Studi di Palermo, Via Archirafi 34, 90123 Palermo, Italy}\\
		\footnotesize{mario.galici@unipa.it}}

	\begin{abstract}
			In this article we study extensions of Steiner triple systems by means of the associated Steiner loops. We recognize that the set of Veblen points of a Steiner triple system corresponds to the center of the Steiner loop. We investigate extensions of Steiner loops, focusing in particular on the case of Schreier extensions, which provide a powerful method for constructing Steiner triple systems containing Veblen points. 
	\end{abstract}
	
	\section*{Introduction}
	Although it has been known since the 50's of the last century (\cite{hale}, \cite{bruckbook}) that any Steiner triple system can be associated with a commutative loop, an extension theory for Steiner triple systems has not been explored. 
	 
	 In this paper, we study Steiner triple systems by means of their associated Steiner loops, using a classic algebraic technique. Indeed, we reduce their structure to that of suitable normal subloops and the corresponding quotient loops. Currently, a general extension theory for loops does not exist. The lack of associativity makes the situation less controllable and very different from the theory of group extensions. Many authors have contributed to this field, working on specific kinds of extensions for some classes of loops. Examples include nuclear extensions \cite{drapalvojtconstruction}, \cite{kinyonphillipsvojt}, Schreier extensions \cite{NagyStrambach}, or the well known Chein's extensions for Moufang loops \cite{Chein} (see also \cite{CheinPflugfelderSmith}, \cite{DharwadkerSmith}, \cite{nagyextinvolbol}, \cite{EilenbergMacLane}, \cite{JohnsonLeedhamgreen}, \cite{GaliciNagy}).
	 
	 This paper primarily focuses on Schreier extensions of Steiner loops, which offer a method for constructing and classifying specific Steiner triple systems which present similarities to the point-line designs of projective spaces over $\mathrm{GF}(2)$. Indeed, we distinguish the case where the normal subloop is central: after showing that central elements correspond to \emph{Veblen points} (see \Cref{defveblen}), we introduce an extension theory inspired by the cohomology theory for commutative groups. This specific theory provides a constructive approach to describe Steiner triple systems containing Veblen points. In particular, the set of Veblen points turns out to be a Steiner triple subsystem of size $2^n-1$ which is the point-line design of a projective space over the field $\mathrm{GF}(2)$. The whole Steiner loop, in this case, is a Schreier extension of its center by the corresponding quotient loop, and it can be described by a factor system $f$ as in \Cref{lemmatriples}. In particular, in \cite{FilipponeGalici} this method is applied to classify (respectively, enumerate) Steiner triple systems of order $19$ (resp. of order $27$ and $31$), containing Veblen points. In \Cref{maxnumbveblenpoints}, we set a threshold for the maximum number of Veblen points in a non-projective Steiner triple system, culminating in \Cref{numberveblenpg}: if the order of a Steiner triple system is between $2^{n}-1$ and $2^{n+1}-1$, and if it contains at least $2^{n-3}$ Veblen points, then it is isomorphic to $\mathrm{PG}(n-1,2)$.
	 
	 In \Cref{cohomology}, we face the problem of defining equivalent and isomorphic extensions.  Examples are given. 
	
	Recursive methods for constructing "products" of Steiner triple systems are well known \cite[Ch. 3]{ColbournRosa}. However, only one of these, the so-called \emph{doubling} construction, coincides with the extension provided by our more general construction, using a \emph{Steiner operator}, in the trivial case where the subloop has index two. Indeed, such a subloop is necessarily  normal and corresponds to a \emph{projective hyperplane}, a topic first studied by Teirlinck \cite{teirlinck} and later by Doyen, Hubaut, and Vandensavel \cite{DHV}. In the latter case, the problem of classifying $\mathrm{STS}(v)$s containing a projective hyperplane is reduced, into the present frame, to classifying $\mathrm{STS}(\frac{v-1}{2})$s and symmetric Latin squares on $\frac{v-1}{2}$ letters with a fixed element in the main diagonal.

	\section{General facts}\label{sec:generalfacts}

	A \emph{Steiner triple system} ($\mathrm{STS}$ for short) $(\mathcal{S},\mathcal{T})$ consists of a set $\mathcal{S}$ of $v$ elements (\emph{points}) and a family $\mathcal{T}$ of $3$-subsets of $\mathcal{S}$, called \emph{triples} (also \emph{blocks} or \emph{lines}), with the property that every $2$-subset of $\mathcal{S}$ occurs in exactly one triple of $\mathcal{T}$. The size $v$ of the set $\mathcal{S}$ is called the \emph{order} of the Steiner triple system. Throughout this paper we will refer to a Steiner triple system $(\mathcal{S},\mathcal{T})$ simply by its set of points $\mathcal{S}$. For a general reference we address the reader to \cite{ColbournDinitz} and \cite{ColbournRosa}.
	
	A Steiner triple system of order $v$ (denoted by $\mathrm{STS}(v)$) exists if and only if $v\equiv 1, 3 \pmod 6$, and in this case, $v$ is said \emph{admissible}. The total number of blocks of an $\mathrm{STS}(v)$ is $b=\frac{v(v-1)}{6}$. Note that we also consider the trivial cases of the $\mathrm{STS}(1)$ with one point and no blocks, and of the $\mathrm{STS}(3)$ with three points and a unique triple. Classic examples of Steiner triple systems include the \emph{projective} and the \emph{affine} systems. The former are the point-line designs of projective spaces $\mathrm{PG}(n,2)$ over $\mathrm{GF}(2)$, while the latter are the point-line designs of affine spaces $\mathrm{AG}(n,3)$  over $\mathrm{GF}(3)$.

	A \emph{quasigroup} is a set $L$ endowed with a binary operation $x\cdot y$ such that the equations $a\cdot x=b$ and $y\cdot a=b$ have unique solutions $x$, and $y$. The solutions are denoted by divisions on the left, and on the right, $x=a\backslash b$, and $y=b/a$. \emph{Loops} are quasigroups with an identity element $\Omega$. The multiplication sign is often ignored, whereas $(x\cdot y)\cdot z$ is written as $xy\cdot z$. A loop operation does not need to be associative: if associativity holds, the loop is actually a group. For a general reference, we address the reader to \cite{pflugfelder} and \cite{bruckbook}.
	
	A loop homomorphism is a map $\varphi\colon L_1\rightarrow L_2$ between two loops such that $\varphi(xy)=\varphi(x)\varphi(y)$ for all $x,y\in L_1$. Loop homomorphisms preserve left and right division, as well as the identity element. A subloop $N\leq L$ is \emph{normal} if it is the kernel of a homomorphism or, equivalently, if the  relations
	\begin{equation*}
		x N=N x, \quad x\cdot N y=x N\cdot y, \quad x\cdot y N=x y\cdot N,
	\end{equation*}
	hold for any $x,y\in L$. If $L$ is commutative, the three normality conditions reduce to $x\cdot y N=x y \cdot N$. 

	The \emph{center} of a loop $L$, i.e.
	\begin{equation*}
		Z(L)=\{z\in L \mid za=az, \ a\cdot zb=az\cdot b=z\cdot ab, \text{ for all } a,b\in L\}
	\end{equation*}
	is always a normal subloop in $L$, which actually is a subgroup. If $L$ is commutative, it is simply the subgroup  
		\begin{equation*}
		Z(L)=\{z\in L \mid \ a\cdot zb=az\cdot b=z\cdot ab, \text{ for all } a,b \in L\}.
	\end{equation*}
	
	A loop $L$ is said \emph{totally symmetric} if it is commutative and the identity 
	\begin{equation}
		x\cdot xy=y
	\end{equation}
	holds for every $x,y\in L$. In a totally symmetric loop, the left and right inverses $x\backslash \Omega$ and $\Omega/x$ of each element $x$ coincide; we call this the \emph{inverse} of $x$, denoted by $x^{-1}$. Every totally symmetric  loop has exponent two. For loops of exponent two, the totally symmetric property $x\cdot xy=y$ is equivalent to the weak associativity which says that $x \cdot yz=\Omega$ whenever $xy \cdot z=\Omega$, for any $x,y,z\in L$.

	\section{Steiner loops}
	
	Steiner loops are precisely the finite totally symmetric loops defined in \Cref{sec:generalfacts}. In fact, as early as 1958, Bruck observed in \cite{bruckbook} that a totally symmetric loop is essentially the algebraic version of a Steiner triple system. Here we recall the definition.
	
	\begin{definition}\label{steinerloopdef}
		Consider a Steiner triple system $\mathcal{S}$ and let $\Omega$ be a further element not belonging to $\mathcal{S}$. We define $\mathcal{L_S}$ as the set $\mathcal{S} \cup\{{\Omega}\}$ endowed with the binary operation $\cdot $ described as follows:
		\begin{itemize}
			\item for any distinct $x$ and $y$ in $\mathcal{S}$, the product $x \cdot y$ is defined as the third point in the triple of $\mathcal{S}$ containing $x$ and $y$;
			\item for any $x \in \mathcal{L_S}$, we set $x^2=\Omega$ and $x \cdot \Omega=\Omega \cdot x=x$.
		\end{itemize}
		$\mathcal{L_S}$ is called a \emph{Steiner loop}.
	\end{definition}
	 These concepts have been studied, for example, in \cite{quackenbush}, \cite{Strambach}, \cite{steinerloopsmougangth}, \cite{nilpotentloopsclass2}, \cite{hale}, and \cite{orientedsts}.  A Steiner loop $\mathcal{L_S}$ turns out to be a group, specifically an elementary abelian 2-group, precisely when the Steiner triple system $\mathcal{S}$ is projective (\cite{dipaolagroup}, see also \cite{burattinakic}). 
	 
	 \begin{remark}
	 	Actually, given a Steiner triple system, there are (at least) two ways to define an operation that gives a loop structure. In the case of \Cref{steinerloopdef}, the unit is an additional element, while in the other constructions the identity element is a chosen point within the Steiner triple system. Both kinds of loops already appear in \cite{Chein}, and one of the second type is the object of \cite{steinerloopsaffine}.
	 	
	 \end{remark}

	Let $\varphi$ be a homomorphism between two Steiner loops $\mathcal{L}_{\mathcal{S}_1}$ and $\mathcal{L}_{\mathcal{S}_2}$ with identities $\Omega_1$ and $\Omega_2$. If $\{a_1,a_2,a_3\}$ is a triple of $\mathcal{S}_1$ then, since $a_1a_2a_3=\Omega_1$, either $\{\varphi(a_1),\varphi(a_2),\varphi(a_3)\}$ is a triple of $\mathcal{S}_2$ or $\varphi(a_i)=\Omega_2$ and $\varphi(a_j)=\varphi(a_k)$, with $\{i,j,k\}=\{1,2,3\}$. 
	
	If $\mathcal{S}_1$ and $\mathcal{S}_2$ have the same order, then the isomorphisms of loops $\mathcal{L}_{\mathcal{S}_1}\leftrightarrows\mathcal{L}_{\mathcal{S}_2}$ correspond exactly to the isomorphisms of Steiner triple systems $\mathcal{S}_1\leftrightarrows\mathcal{S}_2$. Naturally, the group $\mathrm{Aut}(\mathcal{L_S})$ can be identified with $\mathrm{Aut}(\mathcal{S})$.
	
	Moreover, there is a one-to-one correspondence between subloops of $\mathcal{L_S}$ and Steiner triple subsystems of $\mathcal{S}$, and, if a subloop is normal, the quotient loop gives in turn another Steiner triple system, as shown in \Cref{Correspondence}.

		\begin{theorem}\label{Correspondence} 
		Let $\mathcal{S}$ be a Steiner triple system and $\mathcal{L_S}$ the corresponding Steiner loop with identity $\Omega$.
		\begin{itemize}
			\item[i)] $\mathcal{L}'$ is a subloop of $\mathcal{L_S}$ if and only if it is the Steiner loop $\mathcal{L_R}$ associated with a subsystem $\mathcal{R}$ of $\mathcal{S}$.
			\item[ii)] If $\mathcal{L_N}$ is a normal subloop of $\mathcal{L_S}$, then
			each nontrivial coset $x \mathcal{L_N}$ generates a subsystem of $\mathcal{S}$ containing $\mathcal{N}$.
			\item[iii)] If $\mathcal{L_N}$ is a normal subloop of $\mathcal{L_S},$ then the factor loop $\mathcal{L_S}/\mathcal{L_N}$ is a Steiner loop $\mathcal{L_Q}$, where $\mathcal{Q}$ is the Steiner triple system consisting of the nontrivial cosets of $\mathcal{L_N}$.
		\end{itemize}
	\end{theorem}
	
	\begin{proof}
		\begin{itemize}
			\item[i)] $\mathcal{L}'$ is a subloop of $\mathcal{L_S}$ if and only if it is closed under the operation of $\mathcal{L_S}$, which is equivalent to saying that if two different elements of $\mathcal{S}$ are contained in $\mathcal{R}:=\mathcal{L}'\setminus\{\Omega\}$, then the third point $z=xy$ of the triple through $x$ and $y$ is also in $\mathcal{R}$.
				
			\item[ii)]
			Let $\mathcal{L_N}$ be a normal subloop of $\mathcal{L_S}$ and $x\notin\mathcal{L_N}$.
			The nontrivial coset $x\mathcal{L_N}$ generates the Steiner triple subsystem $x\mathcal{L_N}\cup \mathcal{N}$. 
				
			In fact, if $x n_1$ and $x n_2$ are two distinct elements of $x \mathcal{L_N}$, then $(x n_1) \cdot (x n_2)= n_3 \in \mathcal{L_N}$, since $(x \mathcal{L_N}) \cdot (x \mathcal{L_N})=\mathcal{L_N}$. This means that the third point in the triple through $x n_1$ and $x n_2$ is an element of $\mathcal{N}$. Hence, the subsystem of $\mathcal{S}$ generated by $\mathcal{N}$ must be contained in $x\mathcal{L_N}\cup \mathcal{N}$.
			Moreover, if $x n_1 \in x\mathcal{L_N}$ and $n_2 \in \mathcal{N}$, then we have that $(x n_1) \cdot n_2=x n_3$ since $(x \mathcal{L_N}) \cdot \mathcal{L_N}= x \mathcal{L_N}$, that is	$\{x n_1, n_2, x n_3 \}$ is a triple $\mathcal{S}$ contained in $\mathcal{N} \cup x \mathcal{L_N}$. 
			This proves that $x\mathcal{L_N}$ generates the subsystem $x\mathcal{L_N}\cup \mathcal{N}$. Note that the order of $\mathcal{N} \cup x \mathcal{L_N}$ is admissible: indeed, if $w$ is the size of $\mathcal{N}$, then $|\mathcal{N} \cup x \mathcal{L_N}|=2w+1 \equiv 3$ or $1\pmod 6$ whenever $w \equiv 1$ or $3\pmod 6$, respectively.
				
			\item[iii)] The last assertion follows from the fact that the quotient $\mathcal{L_S}/\mathcal{L_N}$ is a finite totally symmetric loop.
		\end{itemize}
	\end{proof}
	
	If $\mathcal{L_N}$ is a normal subloop of $\mathcal{L_S}$ and $\mathcal{L_Q}$ is the corresponding quotient loop, we say that $\mathcal{N}$ is a \emph{normal subsystem} of $\mathcal{S}$ and $\mathcal{Q}$ is the corresponding \emph{quotient system}.  Of course, normal subsystems correspond to kernels of epimorphisms.

		\begin{definition}
		Let $v$ be the order of a Steiner triple system. We say that $v+1=(u+1)(w+1)$ is an \emph{admissible factorization} if $u$ and $w$ are admissible in the sense of Steiner triple systems. 	
	\end{definition}
	
	\begin{example}
		Since the factorization $14=2 \cdot 7$ is not admissible, we can say that the two non-isomorphic $\mathrm{STS}(13)$s cannot have nontrivial normal subsystems, or equivalently, the corresponding Steiner loops are simple.
	\end{example}
	
	A class of subsystems that are always normal is that of \emph{projective hyperplanes} which were studied first by Teirlinck \cite{teirlinck} and later by Doyen, Hubaut, and Vandensavel \cite{DHV}.
	
	\begin{definition}
		A proper subsystem $\mathcal{N}$ of $\mathcal{S}$ is called a \emph{projective hyperplane} if every triple of $\mathcal{S}$ has a nonempty intersection with $\mathcal{N}$.
	\end{definition}

	Equivalently, a subsystem $\mathcal{N}$ of an
	$\mathrm{STS}(v)$ is a projective hyperplane if and only if $|{\mathcal N}|=\frac{v-1}{2}$. Naturally, projective hyperplanes correspond exactly to subloops of index $2$, which are always normal. 
	
	On the other hand, when we raise the index to $4$, a subloop of a Steiner loop $\mathcal{L_S}$ is not necessarily normal, as shown in the following example. 
	
	\begin{example}
		Let $\mathcal{S}$ be the $\mathrm{STS}(15)$ with triples given by the columns of \Cref{tab:sts15.2}. In the classification of the $80$ non-isomorphic Steiner triple systems of order $15$ in \cite{ColbournDinitz}, pp. 31-33, it is presented as $\# 2$.
		\begin{table}[H]
			\centering
			\begin{tabular}{ccccccccccccccccccccccccccccccccccc}
			$0 $\!\! & \!\!$ 0 $\!\! & \!\!$ 0 $\!\! & \!\!$ 0 $\!\! & \!\!$ 0 $\!\! & \!\!$ 0 $\!\! & \!\!$ 0 $\!\! & \!\!$ 1 $\!\! & \!\!$ 1 $\!\! & \!\!$ 1 $\!\! & \!\!$ 1 $\!\! & \!\!$ 1 $\!\! & \!\!$ 1 $\!\! & \!\!$ 2 $\!\! & \!\!$ 2 $\!\! & \!\!$ 2 $\!\! & \!\!$ 2 $\!\! & \!\!$ 2 $\!\! & \!\!$ 2 $\!\! & \!\!$ 3 $\!\! & \!\!$ 3 $\!\! & \!\!$ 3 $\!\! & \!\!$ 3 $\!\! & \!\!$ 4 $\!\! & \!\!$ 4 $\!\! & \!\!$ 4 $\!\! & \!\!$ 4 $\!\! & \!\!$ 5 $\!\! & \!\!$ 5 $\!\! & \!\!$ 5 $\!\! & \!\!$ 5 $\!\! & \!\!$ 6 $\!\! & \!\!$ 6 $\!\! & \!\!$ 6 $\!\! & \!\!$ 6$ \\
			$1 $\!\! & \!\!$ 3 $\!\! & \!\!$ 5 $\!\! & \!\!$ 7 $\!\! & \!\!$ 9 $\!\! & \!\!$ b $\!\! & \!\!$ d $\!\! & \!\!$ 3 $\!\! & \!\!$ 4 $\!\! & \!\!$ 7 $\!\! & \!\!$ 8 $\!\! & \!\!$ b $\!\! & \!\!$ c $\!\! & \!\!$ 3 $\!\! & \!\!$ 4 $\!\! & \!\!$ 7 $\!\! & \!\!$ 8 $\!\! & \!\!$ b $\!\! & \!\!$ c $\!\! & \!\!$ 7 $\!\! & \!\!$ 8 $\!\! & \!\!$ 9 $\!\! & \!\!$ a $\!\! & \!\!$ 7 $\!\! & \!\!$ 8 $\!\! & \!\!$ 9 $\!\! & \!\!$ a $\!\! & \!\!$ 7 $\!\! & \!\!$ 8 $\!\! & \!\!$ 9 $\!\! & \!\!$ a $\!\! & \!\!$ 7 $\!\! & \!\!$ 8 $\!\! & \!\!$ 9 $\!\! & \!\!$ a$ \\
			$2 $\!\! & \!\!$ 4 $\!\! & \!\!$ 6 $\!\! & \!\!$ 8 $\!\! & \!\!$ a $\!\! & \!\!$ c $\!\! & \!\!$ e $\!\! & \!\!$ 5 $\!\! & \!\!$ 6 $\!\! & \!\!$ 9 $\!\! & \!\!$ a $\!\! & \!\!$ d $\!\! & \!\!$ e $\!\! & \!\!$ 6 $\!\! & \!\!$ 5 $\!\! & \!\!$ a $\!\! & \!\!$ 9 $\!\! & \!\!$ e $\!\! & \!\!$ d $\!\! & \!\!$ b $\!\! & \!\!$ c $\!\! & \!\!$ d $\!\! & \!\!$ e $\!\! & \!\!$ c $\!\! & \!\!$ b $\!\! & \!\!$ e $\!\! & \!\!$ d $\!\! & \!\!$ e $\!\! & \!\!$ d $\!\! & \!\!$ c $\!\! & \!\!$ b $\!\! & \!\!$ d $\!\! & \!\!$ e $\!\! & \!\!$ b $\!\! & \!\!$ c$ \\
			\end{tabular}
			\caption{$\mathrm{STS}(15)$ $\# 2$}\label{tab:sts15.2}
		\end{table}
		Any triple $\mathcal{N}$ of $\mathcal{S}$ gives a subloop $\mathcal{L_N}<\mathcal{L_S}$ of index four. Let $\mathcal{N}$ be, for instance, the triple $\{5,9,c\}$. Normality requires that for any $x,y\in\mathcal{L_S}$ and any $n_1\in\mathcal{L_N}$, $x(yn_1)=(xy)n_2$ for some $n_2\in\mathcal{L_N}$. If we choose $x=3$, $y=1$, $n_1=9$, then 
		\begin{equation*}
			x(yn_1)=3(1\cdot 9)=3\cdot 7=b,
		\end{equation*}
		but the equation $(3\cdot 1)n_2=b$, being equivalent to $5\cdot n_2=b$, leads to $n_2=5\cdot b =a$, which is not an element of $\mathcal{L_N}$.
	\end{example}
	
	We want to describe normality of subloops with a combinatorial perspective, especially in small cases. Before going into this, it is important to note that the center $\mathcal{Z}$ of a Steiner loop $\mathcal{L_S}$  has order $2^t$, for some non-negative integer $t$, since it is an elementary abelian $2$-group. Also, since $\mathcal{Z}$ is a normal subloop, its order must divide that of $\mathcal{L_S}$. Thus, if the center is not trivial, then $\mathcal{L_S}$ has an admissible factorization $v+1=2^t(w+1)$ with $t\geq 1$.

	\begin{definition}{\cite[p. 147]{ColbournRosa}}\label{defveblen}
		A point $x$ in a Steiner triple system $\mathcal{S}$ is a \emph{Veblen point} if whenever $\{x,a,b\}$, $\{x,c,d\}$, $\{y,a,c\}$ are triples of $\mathcal{S}$, also $\{y,b,d\}$ is a triple of $\mathcal{S}$.
	\end{definition}
	
 	\Cref{defveblen} says that any two different triples through a Veblen point $x$ produce a Pasch configuration, as in \Cref{fig:pasch_antipasch_conf}.
	
	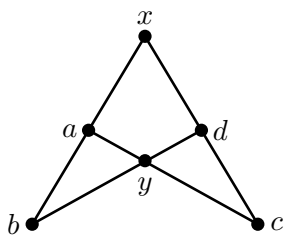
\begin{figure}[H]
		\begin{center}
			\begin{tikzpicture}[scale=0.25] 
				\draw[line width=1] (-6,0) -- (0,10);
				\draw[line width=1] (6,0) -- (0,10);
				\draw[line width=1] (-6,0) -- (3,5);
				\draw[line width=1] (6,0) -- (-3,5);

				\fill (-6,0) circle (10pt);
				\fill (6,0) circle (10pt);
				\fill (0,10) circle (10pt);
				\fill (3,5) circle (10pt);
				\fill (0,3.38) circle (10pt);
				\fill (-3,5) circle (10pt);
				
				\node at (0,11) {$x$};
				\node at (0,2) {$y$};
				\node at (-7,0) {$b$};
				\node at (7,0) {$c$};
				\node at (-4,5) {$a$};
				\node at (4,5) {$d$};
			\end{tikzpicture}\caption{Pasch configuration}
			\label{fig:pasch_antipasch_conf}
		\end{center}
	\end{figure}
	\noindent Equivalently we can say that $x$ is a Veblen point if and only if any two different triples containing $x$ generate an $\mathrm{STS}(7)$. 
	
	\begin{remark}
	Note indeed that, if in a Pasch configuration there is a Veblen point, it generates a Fano plane, but in general it is not true that a Pash configuration is necessarily contained in a Fano plane. A result concerning this topic deals with Steiner loops satisfying \emph{Moufang's theorem}, that is, every three associating elements generate a group. It is known that a Steiner loop is a Moufang loop if and only if it is associative. However, some non-associative Steiner loops satisfy Moufang's theorem. In \cite{steinerloopsmougangth}, the authors proved that a Steiner loop $\mathcal{L_S}$ satisfies Moufang's theorem if and only if every Pasch configuration in $\mathcal{S}$ generates a sub-$\mathrm{STS}(7)$.
	\end{remark}
	
	Veblen points give a characterization of projective Steiner triple systems, which in its original form was one of the \emph{Veblen-Young axioms} for projective spaces (see \cite{veblenyoung}). However, the version we report here was presented by C.J. Colbourn and A. Rosa.
	
	\begin{theorem}{\cite[Th. 8.15]{ColbournRosa}}\label{thmVeblenYoung}
		Let $\mathcal{S}$ be a Steiner triple system of order $v$, and suppose that $2^n\leq v < 2^{n+1}$. The system $\mathcal{S}$ is isomorphic to $\mathrm{PG}(n+1,2)$, and $v=2^{n+1}-1$, if and only if every element of $\mathcal{S}$ is a Veblen point.
	\end{theorem}
	
	\Cref{thmVeblenYoung} will be improved by \Cref{maxnumbveblenpoints}, but first we need to show some results about normality and Veblen points. The next result explicitly characterizes normal subsystems consisting of a singleton, and also gives an algebraic meaning to Veblen points in the context of Steiner loops, showing that they correspond to the center of the loop. 
	
	\begin{theorem}\label{veblencenter}
		Let $\mathcal{S}$ be a Steiner triple system. The following are equivalent.
		\begin{enumerate}[(i)]
			\item The sub-$\mathrm{STS}(1)$ $\mathcal{N}=\{x\}$ is a normal subsystem of $\mathcal{S}$; \label{thm.veb.cond1}
			\item $x$ is a Veblen point of $\mathcal{S}$;\label{thm.veb.cond2}
			\item $x$ is a (nontrivial) central element of $\mathcal{L_S}$. \label{thm.veb.cond3}
		\end{enumerate}
	\end{theorem}
	\begin{proof}
		Let $\mathcal{N}=\{x\}$  be a normal subsystem of $\mathcal{S}$. The normality condition of the corresponding Steiner subloop $\mathcal{L_N}$
		is, in this case,
		\begin{equation*}
			a\cdot b x=ab\cdot x \quad \text{for all } a,b\in\mathcal{L_S}.
		\end{equation*} 
		Hence, \eqref{thm.veb.cond1} and \eqref{thm.veb.cond3} are equivalent. 
		
		Let now $x$ be a Veblen point of $\mathcal{L_S}$. Consider two triples $\{x,a,ax\}$ and $\{x,b,bx\}$. Since $\{a,bx,a\cdot bx\}$ is a triple, then $\{ax,b, a\cdot bx\}$ is a triple as well, which means that $ax\cdot b=a \cdot xb$. 
		This proves the equivalence between \eqref{thm.veb.cond2} and \eqref{thm.veb.cond3}.
	\end{proof}
	
	The following result is a direct consequence of \Cref{veblencenter}, more precisely it follows from the fact that the set of Veblen points of a Steiner triple system is exactly the set of nontrivial central elements of the associated Steiner loop.
	
	\begin{corollary}
		The Veblen points of a Steiner triple system of order $v$ form a normal subsystem of order $2^{n+1}-1\leq v$ isomorphic to $\mathrm{PG}(n,2)$.
	\end{corollary}
	
	We also want to describe the normality of subsystems of order three in combinatorial terms.

	\begin{theorem} Let $\mathcal{S}$ be a Steiner triple system. If a triple is normal, then any outer point together with it generates a Fano plane.
	\end{theorem}	
	\begin{proof}
		Let $\mathcal{N}=\{x,y,xy\}$ be a normal triple of $\mathcal{S}$ and $a$ an outer point. From the normality condition, we know that the solutions $n_1,n_2,n_3$ of the equations
		\begin{equation*}
			a(xy)=(ax)n_1, \quad a(yx)=(ay)n_2, \quad (xa)(ya)=xn_3,
		\end{equation*}
		belong to $\mathcal{L_N}$. It is easy to check that the only possibilities which do not lead to any contradiction are $n_1=y$, $n_2=x$ and $n_3=y$, giving the identities 
		\begin{equation*}
			a(xy)=(ax)y, \quad a(yx)=(ay)x, \quad (xa)(ya)=xy.
		\end{equation*}
		This means that $\mathcal{N}$ and $a$ generate a Fano plane, as shown in \Cref{fig:fanoplaneen}.
		\begin{figure}[H]
			\begin{center}
				\begin{tikzpicture}[scale=0.3]
					\draw[line width=1] (-6,0) -- (6,0) -- (0,10) -- cycle;
					\draw[line width=1] (-6,0) -- (3,5);
					\draw[line width=1] (6,0) -- (-3,5);
					\draw[line width=1] (0,10) -- (0,0);
					\draw[line width=1] (0,3.38) circle [radius=3.38];
					
					\fill (-6,0) circle (10pt);
					\fill (6,0) circle (10pt);
					\fill (0,10) circle (10pt);
					\fill (0,0) circle (10pt);
					\fill (3,5) circle (10pt);
					\fill (0,3.38) circle (10pt);
					\fill (-3,5) circle (10pt);

					\node at (0,11) {$x$};
					\node at (0,-1) {$xa$};
					\node at (1.5,3.38) {$a$};
					\node at (-7.5,0) {$xy$};
					\node at (7.5,0) {$ya$};
					\node at (-4,5) {$y$};
					\node at (5,5) {$axy$};
				\end{tikzpicture}\caption{Fano plane generated by a normal triple and an outer point}\label{fig:fanoplaneen}
			\end{center}
		\end{figure}
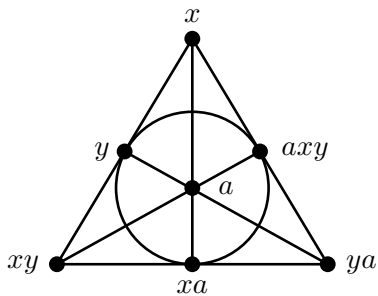
		\vspace{-0.6cm}
		\noindent Thus, the theorem is proved.
	\end{proof}
	
	Using the characterization of Veblen points in \Cref{veblencenter}, we can give a necessary condition on the existence of Steiner triple systems with such points. 
	
	\begin{proposition}\label{number of Veblen points1}
		If $v+1=2^t(w+1)$ is an admissible factorization only for $t=0$, then any $\mathrm{STS}(v)$ contains no Veblen points.
	\end{proposition}
	\begin{proof}
		The claim follows from the fact that, if the center has cardinality $2^t$, then $\frac{v+1}{2^t}$ must be the order of the quotient Steiner loop.
	\end{proof}
	
	After defining Schreier extension, in \Cref{section:schreierext}, we will be able to give a necessary and sufficient condition for the existence of Steiner triple systems of order $v$ with (at least) $2^t-1$ Veblen points, and we will present a constructive method for obtaining all such $\mathrm{STS}$s. 
	
	Now we prove a more general fact about Veblen points, Pasch configurations and Fano planes that can be useful in classifying Steiner triple systems with Veblen points.
	
	\begin{lemma} \label{paschconf} 
		If ${\mathcal S}$ is an $\mathrm{STS}(v)$, then:
		\begin{enumerate}[(i)]
			\item The number of Pasch configurations through a fixed Veblen point is $\frac{(v-1)(v-3)}{4}$.
			
			\item The number of Fano planes containing a fixed Veblen point is $\frac{(v-1)(v-3)}{24}$.
			
			\item If ${\mathcal S}$ has two distinct Veblen points $a,b$, then the third point $c=ab$ in their triple is a Veblen point as well. Moreover, there are $\frac{v-3}{4}$ Fano planes containing the triple $\ell=\{a,b,ab\}$.   
		\end{enumerate}
	\end{lemma}
	
	\begin{proof}
		If $a$ is a Veblen point, then for all points $b \in \mathrm{STS}(v)$, $a \neq b$, there are $\frac{v-3}{4}$ Pasch configurations through 
		$b$ and $a$ which do not contain the block $\{a, b, a b\}$. See \Cref{fig:paschnotab}.
		\begin{figure}[H]
			\begin{center}
					\begin{tikzpicture}[scale=0.25] 
					\draw[line width=1] (-6,0) -- (0,10);
					\draw[line width=1] (6,0) -- (0,10);
					\draw[line width=1] (-6,0) -- (3,5);
					\draw[line width=1] (6,0) -- (-3,5);

					\fill (-6,0) circle (10pt);
					\fill (6,0) circle (10pt);
					\fill (0,10) circle (10pt);
					\fill (3,5) circle (10pt);
					\fill (0,3.38) circle (10pt);
					\fill (-3,5) circle (10pt);
					
					\node at (0,11) {$a$};
					\node at (0,2) {$b$};
				\end{tikzpicture}\caption{Pasch configuration containing $a$ and $b$ but not their triple}\label{fig:paschnotab}
			\end{center}
		\end{figure}
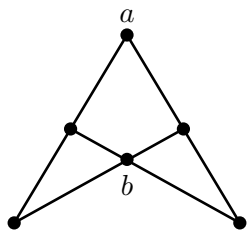
		This follows from the fact that we cannot choose $ab$ to be in the Pasch configuration, so we are left with $v-3$ points of the $\mathrm{STS}(v)$. In a Pasch configuration there are four further points other than $a$ and $b$. Fixing one of these four points, the others are uniquely determined: indeed, if we fix $x$ to be in the configuration, the others must necessarily be $ax$, $bx$, and $a(bx)=(ax)b$, and rearranging these four points we obtain the same configuration. Finally, since  the point $b$ can be chosen in $v-1$ different ways, we get the first assertion. 
		
		The second assertion follows from the fact that every sub-Fano plane containing the Veblen point $x$ is obtained by $6$ Pasch configurations.\footnote{Note that the number $\frac{(v-1)(v-3)}{24}$ is in fact an integer. Indeed, $v$ can be $1,3,7,9\pmod{12}$: if $v\equiv 1,3\pmod{12}$, one between $v-1$ and $v-3$ is divisible by $12$ and the other is even; if $v\equiv 7,9\pmod{12}$, one between $v-1$ and $v-3$ is divisible by $4$ and the other by $6$.}
		
		Now we proceed to prove the third assertion. Let $a$ and $b$ be two distinct Veblen points of $\mathcal{S}$. Since both points are central, their product $c=ab$ is still central and hence a Veblen point as well. Fixed a Veblen triple $\ell=\{a, b, ab\}$, let $x$ be a point of $\mathcal{S}$ not in $\ell$. Since $\ell$ is a triple of Veblen points, together with $x$ it generates a Fano plane.
		We can choose $x$ in $v-3$ different ways, but replacing it with $ax$, $bx$ or $(ab)x=b(ax)=a(bx)$ we obtain the same Fano plane. Hence there are $\frac{v-3}{4}$ different Fano planes containing the Veblen triple $\ell=\{a, b, ab\}$.
	\end{proof}

	\begin{remark}
	 	To illustrate that these results can simplify the counting of Veblen points in a Steiner triple system, we consider Steiner triple systems of order $15$. We use the classification of the $80$ non-isomorphic $\mathrm{STS}(15)$s and their main properties listed in \cite[Tables 1.28 and 1.29, pp. 30 - 32]{ColbournDinitz}
		
		If $\mathcal{S}$ is an $\mathrm{STS}(15)$ with a Veblen point,  then it contains at least $7$ Fano planes. The only $\mathrm{STS}(15)$s with this many sub-Fano planes are $\# 1$ and $\# 2$.  If $\mathcal{S}$ has more than one Veblen point, then it contains at least one triple of Veblen points. For each of these Veblen points, there are $7$ Fano planes containing it, $3$ of which contain the whole triple. Thus, $\mathcal{S}$ contains at least $3(7-3)+3=15$ Fano planes. The only $\mathrm{STS}(15)$ with this many sub-Fano planes is $\# 1$, that is, $\mathrm{PG}(3,2)$, therefore all of its elements are Veblen points. Moreover, the $\mathrm{STS}(15)$ $\# 2$ has precisely one Veblen point, and it is easy to see that it is the element labeled with $0$ by checking that it is a central element of the Steiner loop.
	\end{remark}

	\section{Schreier extensions of Steiner loops}\label{section:schreierext}
	
	The most general way to construct an extension of Steiner triple systems is via a \emph{Steiner operator}, as we show in \Cref{extensions}. However, in this section, our main focus is on \emph{Schreier extensions}, that provide a powerful tool for constructing and classifying Steiner triple systems containing Veblen points. This method proved to be effective, in fact in \cite{FilipponeGalici} it is applied to Steiner triple systems of order $19$, $27$ and $31$ containing Veblen points.
	
	Let $N$ be a group with identity $\Omega'$ and $Q$ be a loop with identity $\bar\Omega$. Consider a map $T\colon Q \to \mathrm{Aut}(N)$ with $T(\bar\Omega)=\mathrm{Id}$, and $f\colon Q\times Q \to N$ a function with the property $f(P,\bar \Omega)=f(\bar \Omega, P)=\Omega'$, for every $P\in Q$. From now on, we will use the additive notation for $N$ and the multiplicative notation for $Q$.
	The operation
	\begin{equation}
		\left(P,x\right)\circ\left(R,y\right):=\big(PR,f(P,R)+x^{T(R)}+y\big),
	\end{equation}
	defines on $Q\times N$ a loop $L$. The loop $L$ is an extension
	 of $N$ by $Q$ called a \emph{Schreier extension}. Indeed, the loop $L$ contains
	\begin{equation*}
		\overline{N}=\{(\bar \Omega, x)\mid x\in N\}\simeq N
	\end{equation*}
	as a normal subgroup with corresponding quotient loop isomorphic to $Q$. The function $f$ is called a \emph{factor system}, as in the theory of group extensions. 
	
	This construction is very similar to the corresponding definition for groups, but as we said, the lack of associativity makes the context less controllable. 
	\Cref{lemmatriples} gives necessary and sufficient conditions for a Schreier extension of Steiner loops to be a Steiner loop as well.

	\begin{theorem}\label{lemmatriples}
		A Schreier extension of an associative Steiner loop $\mathcal{L_N}$ by a Steiner loop $\mathcal{L_Q}$, defined by functions $T$ and $f$, gives a Steiner loop $\mathcal{L_S}$ if and only if the following hold:
		\begin{enumerate}
			\item $\mathcal{L_N}$ is central;
			\item $T$ is trivial;
			\item $f$ is symmetric and
			\begin{equation}\label{factosystem}
				f(P,Q)=f(P,PQ)=f(Q,PQ), \quad \text{for every } P,Q\in\mathcal{L_Q}.
			\end{equation}
		\end{enumerate}
		In particular, the operation of $\mathcal{L_S}$ becomes
		\begin{equation*}
			\left(P,x\right)\circ\left(Q,y\right)=\big(PQ,x+y  +f(P,Q)\big).
		\end{equation*}
	\end{theorem}
	
	\begin{proof}
		By Proposition 3.2 in \cite{Schreierloop}, the map $T$ is trivial, $\mathcal{L_N}$ is a central subgroup of $\mathcal{L_S}$ and $f$ is symmetric.	In the resulting loop, the totally symmetric property
		\begin{equation*}
			(P,x)\circ\big((P,x)\circ(Q,y)\big)=(Q,y),
		\end{equation*}
		is equivalent to
		\begin{equation}\label{eq:totsym}
			\big(Q,y+f\left(P,Q\right)+f\left(P,PQ\right)\big)=\left(Q,y\right).
		\end{equation} 
		\Cref{eq:totsym} holds if and only if $f(P,Q)=f(P,PQ)$ for every $P,Q\in \mathcal{Q}$. Of course, by the symmetry of the factor system, $f(P,Q)$ coincides also with $f(Q,PQ)$. 
	\end{proof}

	Schreier extensions are used, for example, in the study of \emph{oriented} Steiner triple systems \cite{orientedsts} and nilpotent Steiner loops of class $2$  \cite{nilpotentloopsclass2}.
	
	From now on, by Schreier extensions of Steiner loops, denoted by a short exact sequence
	\begin{equation}
		\Omega'\longrightarrow \mathcal{L_N} \longrightarrow \mathcal{L_S} \longrightarrow \mathcal{L_Q} \longrightarrow \bar\Omega,
	\end{equation}
	we will mean Schreier extensions satisfying the conditions of \Cref{lemmatriples}. In this situation, we also say that the Steiner triple system $\mathcal{S}$ is a Schreier extension of $\mathcal{N}$ by $\mathcal{Q}$.
	
	\Cref{factosystem} can be reformulated by saying that
	\begin{equation}\label{eq:factoromega}
		f(P,P)=f(P,\bar\Omega)=\Omega' \text{ for every } P\in\mathcal{L_Q},
	\end{equation}
	and that $f$ is constant on the triples of $\mathcal{Q}$, that is,
	\begin{equation}\label{eq:factortriple}
		f(P,Q)=f(P,R)=f(Q,R)  \text{ whenever } \{P,Q,R\} \text{ is a triple of } \mathcal{Q}.
	\end{equation}
	
	We want to stress the fact that since $\mathcal{L_N}$ is in the center of $\mathcal{L_S}$, the elements of $\mathcal{N}$ are Veblen points of $\mathcal{S}$, that is, a Schreier extension of Steiner loops is an extension of a projective geometry $\mathcal{N}$ obtained by a symmetric function which is constant of the triples of $\mathcal{Q}$.
	
	Now we give an example of a Schreier extension resulting in the Steiner triple system of order $15$ with precisely one Veblen point. 
	
	\begin{example} 
		Consider the $\mathrm{STS}(7)$ $\mathcal{Q}$ with points and triples as in  \Cref{fig:fanoquotient}.
		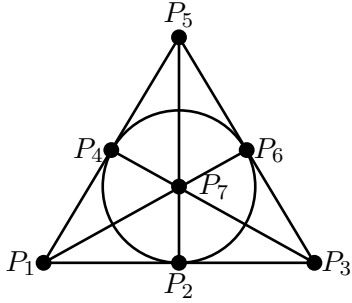
\begin{figure}[H]
			\begin{center}
				\begin{tikzpicture}[scale=0.3]
					\draw[line width=1] (-6,0) -- (6,0) -- (0,10) -- cycle;
					\draw[line width=1] (-6,0) -- (3,5);
					\draw[line width=1] (6,0) -- (-3,5);
					\draw[line width=1] (0,10) -- (0,0);
					
					\fill (-6,0) circle (10pt);
					\fill (6,0) circle (10pt);
					\fill (0,10) circle (10pt);
					\fill (0,0) circle (10pt);
					\fill (3,5) circle (10pt);
					\fill (0,3.38) circle (10pt);
					\fill (-3,5) circle (10pt);
					
					\draw[line width=1] (0,3.38) circle [radius=3.38];
					
					\node at (0,11) {$P_5$};
					\node at (0,-1) {$P_2$};
					\node at (1.55,3.38) {$P_7$};
					\node at (-7,0) {$P_1$};
					\node at (7,0) {$P_3$};
					\node at (-4,5) {$P_4$};
					\node at (4,5) {$P_6$};
				\end{tikzpicture}
			\caption{$\mathrm{STS}(7)$ $\mathcal{Q}$}
			\label{fig:fanoquotient}
		\end{center}
		\end{figure}
		Let $\mathcal{L_S}$ be the Schreier extension of $\mathcal{L_N}=\{\Omega',1\}$ by $\mathcal{L_Q}$ given by the factor system $f$ with
		\begin{equation*}
			f(P_3,P_5)=f(P_3,P_6)=f(P_5,P_6)=1,
		\end{equation*}
		\begin{equation*}
			f(P_3,P_4)=f(P_3,P_7)=f(P_4,P_7)=1,
		\end{equation*}
		and $f(P,Q)=\Omega'$ elsewhere. The element $(\bar\Omega,1)$ is a Veblen point of $\mathcal{S}$. It is the only one because $\mathcal{L_S}$ is not associative. Indeed, for example
			\begin{align*}
				\Big((P_1,\Omega ')\circ(P_2,\Omega ')\Big)\circ(P_4,\Omega ')&=(P_7,1),\\
				(P_1,\Omega ')\circ\Big((P_2,\Omega ')\circ(P_4,\Omega ')\Big)&=(P_7,\Omega ').
		\end{align*}
		Hence, $\mathcal{S}$ is the $\mathrm{STS}(15)$ $\#2$.
	\end{example}

	The theory of Schreier extensions offers a constructive method to obtain Steiner triple systems containing Veblen points. On the other hand, if $\mathcal{L_N}$ is a central subgroup of $\mathcal{L_S}$, then $\mathcal{L_S}$ can be obtained as a Schreier extension of $\mathcal{L_N}$ by $\mathcal{L_Q}=\mathcal{L_S}/\mathcal{L_N}$ (cf. \cite{bruck1}, p. 334). This means that any Steiner triple system $\mathcal{S}$ containing Veblen points can be seen as a Schreier extension of the projective $\mathrm{STS}$ consisting of its Veblen points (or a proper subsystem of it).
	This gives a necessary and sufficient condition on an admissible positive integer $v$ for the existence of a Steiner triple system of order $v$ containing (at least) a certain number of Veblen points. A trivial case occurs when $v=2^t-1$ for $t>0$. In fact, projective geometries $\mathrm{PG}(t-1,2)$ are the only examples $\mathrm{STS}(v)$s with $v$ Veblen points.
	
	\begin{theorem}\label{number of Veblen points2}
		Let $v>1$ be an admissible integer. There exists an $\mathrm{STS}(v)$ with $2^t-1<v$ Veblen points, $t> 0$, if and only if $\frac{v+1}{2^t}\equiv 2,4 \pmod 6$. 
	\end{theorem}
	\begin{proof}
		One direction of the claim follows from the fact that, if the center of a Steiner loop has cardinality $2^t$, then $\frac{v+1}{2^t}$ must be the cardinality of the quotient projective Steiner loop. The other direction is true because we can always construct a Steiner loop $\mathcal{L_S}$ with nontrivial center by considering a Schreier extension of an elementary abelian 2-group $\mathcal{L_N}$ of cardinality $2^t$ by a Steiner loop $\mathcal{L_Q}$ of order $\frac{v+1}{2^t}$.
	\end{proof}
	
	\begin{corollary} 
		Let $v>1$ be an admissible integer such that $v\equiv 1,9 \pmod{12}$. No $\mathrm{STS}(v)$ can have Veblen points.  		
	\end{corollary}
	\begin{proof} Let $v>1$ be an admissible integer. Therefore, $v\equiv 1,3,7,9 \pmod{12}$. By \Cref{number of Veblen points2}, the existence of an $\mathrm{STS}(v)$ with at least one Veblen point is guaranteed if $\frac{v+1}{2}\equiv 2,4 \pmod 6$. This condition is equivalent to $v\equiv 3,7 \pmod{12}$. Hence, any $\mathrm{STS}(v)$ with $v\equiv 1,9 \pmod{12}$ cannot have Veblen points. 
	\end{proof}

		In particular, the first 10 admissible integers $v$ for which any $\mathrm{STS}(v)$ cannot have Veblen points are
		\begin{equation*}
			9,\ 13,\ 21,\ 25,\ 33,\ 37,\ 45,\ 49,\ 57,\ 61.
		\end{equation*}
		Regarding the orders not mentioned, \Cref{number of Veblen points2} shows, for example, that any $\mathrm{STS}(19)$ and any $\mathrm{STS}(27)$ can have at most one Veblen point. Furthermore, an $\mathrm{STS}(31)$ can have $1$, $3$, $7$, $15$, or, as stated in \Cref{thmVeblenYoung}, $31$ Veblen points if it is projective. However, we will see that it cannot have precisely $7$ nor $15$ Veblen points. In fact, there is a crucial threshold: as a consequence of \Cref{theoindex4}, exceeding a certain number of Veblen points forces an $\mathrm{STS}(v)$ to be projective, hence to have $v$ Veblen points. 
	
	\begin{theorem} \label{theoindex4}
		If a Schreier extension of Steiner loops
		\begin{equation}
			\Omega'\longrightarrow \mathcal{L_N} \longrightarrow \mathcal{L_S} \longrightarrow \mathcal{L_Q} \longrightarrow \bar\Omega,
		\end{equation}
		has index at most $4$, then the resulting Steiner triple system $\mathcal{S}$ is projective.
	\end{theorem}
	
	\begin{proof}
		Since the factor loop $\mathcal{L_Q}$ is of order less than or equal to $4$, it can be either the elementary abelian $2$-group of order $2$ or $4$.
		We want to prove that the associative property
		\begin{equation}\label{eq:associativefactorsystem}
			(P, x) \circ ((Q, y) \circ (R,z))=((P, x) \circ (Q, y)) \circ (R,z)
		\end{equation}
		holds for every $P,Q,R\in\mathcal{L_Q}$ and $x,y,z\in\mathcal{L_N}$. Let $f$ be the factor system defining the Schreier extension. 
		If $\mathcal{L_Q}$ has order $2$, then $f$ is the null function and $\mathcal{L_S}$ is a group. Let now $\mathcal{L_Q}$ be the elementary abelian $2$-group of order $4$.	
		\Cref{eq:associativefactorsystem} is equivalent to the property that for every $P,Q,R\in\mathcal{L_Q}$,
		\begin{equation} \label{assocfacsyst} 
			f(P,Q R) + f(Q, R)= f(PQ, R) + f(P,Q).
		\end{equation}
		\begin{itemize}
			\item If the three points form the only triple in the underlying $\mathrm{STS}(3)$ $\mathcal{Q}$, then by \Cref{eq:factortriple} we get \Cref{assocfacsyst}. 
			\item If one out of the three points is the identity element, say $P=\bar \Omega$ without loss of generality, \Cref{assocfacsyst} reduces to
			\begin{equation}
				f(\bar \Omega,Q R)+ f(Q, R)=f(Q, R) +f(\bar \Omega,Q),
			\end{equation}
			which is true since $f(\bar \Omega,Q R)=\Omega'=f(\bar\Omega,Q)$ by \Cref{eq:factoromega}. 
			\item If two out of the three points coincide, say $P=Q$ without loss of generality, then \Cref{assocfacsyst} reduces to 
			\begin{equation}
				f(P,P R)+ f(P,R)= f(\bar \Omega, R)+ f(P,P),
			\end{equation}
			which holds since both sides are equal to $\Omega'$.
		\end{itemize}
		This proves the assertion. 
	\end{proof}
	
	As a direct consequence, \Cref{theoindex4} allows us to strengthen \Cref{thmVeblenYoung}, setting a threshold for the maximum number of Veblen points in a non-projective Steiner triple system.
	
	\begin{corollary}\label{maxnumbveblenpoints}
		Let $\mathcal{S}$ be a Steiner triple system of order $v\geq 7$. $\mathcal{S}$ is projective if and only if it has more than $\frac{v-7}{8}$ Veblen points.
	\end{corollary}
	
	\begin{proof}
		If an $\mathrm{STS}(v)$ $\mathcal{S}$ has more than $\frac{v-7}{8}$ Veblen points, then the center of $\mathcal{L_S}$ has order more than $\frac{v+1}{8}$. Thus, it has index at most $4$ and, by \Cref{theoindex4}, $\mathcal{S}$ is projective.
	\end{proof}
	
	 \Cref{maxnumbveblenpoints} deals specifically with Steiner triple systems of order $v\geq 7$. Note, however, that Steiner triple systems of order $1$ and $3$, which are the only cases left out of the result, are inherently projective.
	Another similar consequence of \Cref{theoindex4} is given by the following result.
	
	\begin{corollary}\label{numberveblenpg}
		If a Steiner triple system of order $v$, with $2^{n}-1\leq v< 2^{n+1}-1$ and $n>1$, contains at least $2^{n-3}$ Veblen points, then it is isomorphic to $\mathrm{PG}(n-1,2)$.
	\end{corollary}
	\begin{proof}
		If we assume that $\mathcal S$ has at least $2^{n-3}$ Veblen points, then, since they form a projective subsystem, the number of Veblen points is at least $2^{n-2}-1$. Thus the center $Z(\mathcal{L_S})$ of $\mathcal{L_S}$ has at least $2^{n-2}$ elements. If $v< 2^{n+1}-1$, then we have $|\mathcal{L_S}/Z(\mathcal{L_S})|<\frac{2^{n+1}}{2^{n-2}}=8$. So, $\mathcal{S}$ is projective and by cardinality reasons it is isomorphic to $\mathrm{PG}(n-1,2)$.
	\end{proof}

	\begin{remark} In view of the last two results, we can easily see the fact that the only $\mathrm{STS}(15)$ with more than one Veblen point is $\mathrm{PG}(3,2)$, as we showed before. Actually, now we can say even more: the only $\mathrm{STS}(v)$s with $v<31$ having more than one Veblen point are $\mathrm{PG}(1,2)$, $\mathrm{PG}(2,2)$ and $\mathrm{PG}(3,2)$. 
	\end{remark}

	\section{Equivalent and isomorphic extensions}\label{cohomology}
	
	In this section we study the concepts of equivalent and isomorphic Schreier extensions of Steiner loops. To stay constructive and not be too theoretical, fixed $\mathcal{L_N}$ and $\mathcal{L_Q}$, we identify each Schreier extension with the corresponding factor system
	\begin{equation*}
		f\colon \mathcal{L_Q}\times\mathcal{L_Q}\longrightarrow\mathcal{L_N}.
	\end{equation*}
	We denote the set of all factor systems with $	\mathrm{Ext_S}(\mathcal{L_N},\mathcal{L_Q})$. We recall that the output values of factor systems are in $\mathcal{L_N}$, which is associative. Therefore, $\mathrm{Ext_S}(\mathcal{L_N},\mathcal{L_Q})$ is a group under the addition of functions. After  \Cref{eq:factortriple}, a factor system is completely determined by the values it takes on the triples of the quotient system $\mathcal{Q}$, thus it is easy to see that the total number of possible Schreier extensions of $\mathcal{L_N}$ by $\mathcal{L_Q}$ is 
	\begin{equation}
		|\mathrm{Ext_S}(\mathcal{L_N},\mathcal{L_Q})|=2^{tb},
	\end{equation}
	where $b$ is the number of triples of $\mathcal{Q}$ and $2^{t}$ is the cardinality of the elementary abelian $2$-group  $\mathcal{L_N}$.  A relevant subgroup of $\mathrm{Ext_S}(\mathcal{L_N},\mathcal{L_Q})$ is formed by the so-called \emph{coboundaries}, defined by the cohomology operator $\delta^1$: if $\varphi$ is a map $\mathcal{L_Q}\longrightarrow \mathcal{L_N}$ sending $\Omega'\longmapsto \bar \Omega$, then $\delta^1 \varphi$ is the function $\mathcal{L_Q}\times \mathcal{L_Q}\longrightarrow\mathcal{L_N}$ defined by
	\begin{equation}\label{eq:delta1}
		(\delta^1\varphi)(P,Q):=\varphi(PQ)-(\varphi(P)+\varphi(Q)).
	\end{equation}
	Note that in \Cref{eq:delta1} the parentheses and the minus sign are superfluous, since the Steiner loop $\mathcal{L_N}$ is associative and of exponent $2$; however, for the sake of generality, we decided to present a more general definition in a non-associative context.
	By the definition of the operator $\delta^1$, every coboundary $\delta^1\varphi$ is in fact a factor system. 
Moreover, $\delta^1(\varphi+\psi)=\delta^1\varphi+\delta^1\psi$ holds, hence the set 
	\begin{equation}
		\mathrm{B}^2(\mathcal{L_Q},\mathcal{L_N}):=\left\{\delta^1\varphi \mid \varphi\colon \mathcal{L_Q}\to \mathcal{L_N}, \  \varphi(\Omega')=\bar\Omega  \right\}
	\end{equation}
	is a subgroup of $\mathrm{Ext_S}(\mathcal{L_N},\mathcal{L_Q})$. Since $\delta^1$ is additive, two coboundaries $\delta^1\varphi$ and $\delta^1\psi$ coincide if and only if they differ by a function $g$ such that $\delta^1g$ is zero, that is, a homomorphism $\mathcal{L_Q}\longrightarrow\mathcal{L_N}$. Therefore, the size of this subgroup is 	
	\begin{equation*}
		|\mathrm{B}^2(\mathcal{L_Q},\mathcal{L_N})|=\frac{|\{ \varphi: \mathcal{L_Q} \to \mathcal{L_N} \mid \varphi(\bar \Omega)=\Omega' \}|}{|\mathrm{Hom}(\mathcal{L_Q},\mathcal{L_N})|}=\frac{2^{tw}}{|\mathrm{Hom}(\mathcal{L_Q},\mathcal{L_N})|},
	\end{equation*}
	where $2^t=|\mathcal{L_N}|$ and $w=|\mathcal{Q}|$.
	This subgroup will prove to be useful for classifying factor systems up to \emph{equivalence}. 
	
	\begin{definition}\label{def:equivalent}
		Two Schreier extensions
		\begin{equation*}
			\Omega' \longrightarrow \mathcal{L_N}  {\longrightarrow} 
			\mathcal{L}_{\mathcal{S}_1} {\longrightarrow} \mathcal{L_Q}\longrightarrow\bar\Omega,
		\end{equation*}
		\begin{equation*}
			\Omega' \longrightarrow \mathcal{L_N}  {\longrightarrow} 
			\mathcal{L}_{\mathcal{S}_2} {\longrightarrow} \mathcal{L_Q}\longrightarrow\bar\Omega,
		\end{equation*}
		are said to be \emph{equivalent} if there exists an isomorphism $\mathcal{L}_{\mathcal{S}_1}\longrightarrow\mathcal{L}_{\mathcal{S}_2}$ which induces the identity homomorphism on both $\mathcal{L_N}$ and  $\mathcal{L_Q}$. In this case, the corresponding factor systems $f_1$ and $f_2$ are also called equivalent and we write $f_1\sim f_2$.
	\end{definition}
	
	Saying that an isomorphism $\Phi\colon \mathcal{L}_{\mathcal{S}_1} \to \mathcal{L}_{\mathcal{S}_2}$ induces the identity homomorphism on both $\mathcal{L_N}$ and  $\mathcal{L_Q}$ means that, for every $x\in\mathcal{L_N}$, $P\in\mathcal{L_Q}$, 
	\begin{equation*}
		\Phi(\bar \Omega, x)=(\bar \Omega, x) \quad \text{and} \quad 	\Phi(P, x)=(P, x'),
	\end{equation*}
	for a suitable $x'\in\mathcal{L_N}$. The following result gives a characterization of equivalent factor systems of Schreier extensions of Steiner loops.
	
	\begin{lemma}\label{lemmaequivalence}	
		Two factor systems $f_1,f_2\in \mathrm{Ext_S}(\mathcal{L_N},\mathcal{L_Q})$ are equivalent if and only if they differ by a suitable coboundary  $\delta^1 \varphi$. Moreover, the isomorphism between the corresponding loops realizing the equivalence has the following form:
		\begin{equation}
			(P,x)\longmapsto \left(P, x + \varphi(P)\right).
		\end{equation}
	\end{lemma}	
	
	\begin{proof}
		Let $\Phi\colon \mathcal{L}_{\mathcal{S}_1} \longrightarrow \mathcal{L}_{\mathcal{S}_2}$ be an isomorphism between the Steiner loops corresponding, respectively, to the factor systems $f_1$ and $f_2$. Suppose that $\Phi$ defines an equivalence of extensions. We have
		\begin{align*}
			\Phi(P,x)&=\Phi	\left( (P,\Omega') \circ (\bar\Omega,x)\right)\\
			&=\Phi(P,\Omega')\circ (\bar\Omega,x)\\
			&=(P,\varphi(P))\circ(\bar\Omega,x)\\
			&=(P,x+\varphi(P)),
		\end{align*}
		for a suitable function $\varphi:\mathcal{L_Q}\longrightarrow\mathcal{L_N}$. Since $\Phi(\bar\Omega,\Omega')=(\bar\Omega,\Omega')$, the map $\varphi$ sends $\bar\Omega\mapsto\Omega'$.
		
		Since $\Phi$ is an isomorphism, on the one hand we have
		\begin{align*}
			\Phi\left((P,x)\circ(Q,y)\right)&=\Phi\left(PQ,x+y+f_1(P,Q)\right)\\
			&=(PQ,x+y+f_1(P,Q)+\varphi(PQ)),
		\end{align*}
		on the other hand
		\begin{align*}
			\Phi\left((P,x)\circ(Q,y)\right)&=\Phi\left(P,x\right) \circ \Phi\left(Q,y\right)=(P,x+\varphi(P)) \circ (Q,y+\varphi(Q))\\
			&=(PQ,x+\varphi(P)+y+\varphi(Q)+f_2(P,Q)).
		\end{align*}
		Hence, $f_2=f_1+\delta^1\varphi$.
		
		Conversely, if $f_2=f_1+\delta^1\varphi$, then the function $\Phi(P,x):=(P,x+\varphi(P))$ defines an isomorphism $\Phi\colon \mathcal{L}_{\mathcal{S}_1} \longrightarrow \mathcal{L}_{\mathcal{S}_2}$  which induces the identity both on $\mathcal{L_N}$ and $\mathcal{L_Q}$. Hence, the factor systems $f_1$ and $f_2$ are equivalent.
	\end{proof}
	
	Using the characterization obtained with \Cref{lemmaequivalence}, we find that the number of non-equivalent extensions of $\mathcal{L_N}$ by $\mathcal{L_Q}$ is
	\begin{equation*}	
		\left|\frac{\mathrm{Ext_S}(\mathcal{L_N},\mathcal{L_Q})}{\mathrm{B}^2(\mathcal{L_Q},\mathcal{L_N})}\right|=\frac{2^{tb}}{|\mathrm{B}^2(\mathcal{L_Q},\mathcal{L_N})|}.
	\end{equation*}
	Since $|\mathrm{B}^2(\mathcal{L_Q},\mathcal{L_N})|=2^{wt}/|\mathrm{Hom}(\mathcal{L_Q},\mathcal{L_N})|$ and $b=\frac{w(w-1)}{6}$, this number is
	\begin{equation*}	
		\left|\frac{\mathrm{Ext_S}(\mathcal{L_N},\mathcal{L_Q})}{\mathrm{B}^2(\mathcal{L_Q},\mathcal{L_N})}\right|=2^{tw\left(\frac{w-7}{6}\right)}|\mathrm{Hom}(\mathcal{L_Q},\mathcal{L_N})|.
	\end{equation*}

	\begin{example}\label{nonequivalentsts15}
		An $\mathrm{STS}(15)$ with a Veblen point must be a Schreier extension of the $\mathrm{STS}(1)$ by the $\mathrm{STS}(7)$. Let $\mathcal{L_N}$ be of order $2$ and $\mathcal{Q}$ be the $\mathrm{STS}(7)$. The number of functions $\varphi: \mathcal{L_Q} \to \mathcal{L_N}$ with $\varphi(\bar \Omega)=\Omega'$ is $2^7$, and the order of  $\mathrm{Hom}(\mathcal{L}_\mathcal{Q},\mathcal{L}_\mathcal{N})$ is $2^3$, hence $|\mathrm{B}^2(\mathcal{L_Q},\mathcal{L_N})|=2^4$.
		Therefore, the number of non-equivalent Schreier extensions of $\mathcal{L_N}$ by $\mathcal{L_Q}$ is $\frac{2^7}{2^4}=8$. Actually, we know that among the resulting $8$ Steiner triple systems, we have only $2$ isomorphism classes, since the only $\mathrm{STS}(15)$s with Veblen points are $\#1$ and $\#2$. For this we need a further reduction. 
	\end{example}

	\begin{definition}\label{def:isomorphic}
		Two Schreier extensions
		\begin{equation*}
			\Omega' \longrightarrow \mathcal{L_N}  {\longrightarrow} 
			\mathcal{L}_{\mathcal{S}_1} {\longrightarrow} \mathcal{L_Q}\longrightarrow\bar\Omega,
		\end{equation*}
		\begin{equation*}
			\Omega' \longrightarrow \mathcal{L_N}  {\longrightarrow} 
			\mathcal{L}_{\mathcal{S}_2} {\longrightarrow} \mathcal{L_Q}\longrightarrow\bar\Omega,
		\end{equation*}
		are said to  be \emph{isomorphic} if there is an isomorphism $\mathcal{L}_{\mathcal{S}_1}\longrightarrow\mathcal{L}_{\mathcal{S}_2}$ leaving $\mathcal{L_N}$ invariant.  In this case, the corresponding factor systems $f_1$ and $f_2$ are also called isomorphic and we write $f_1\simeq f_2$.
	\end{definition}
	
	By definition, two isomorphic Schreier extensions again give two isomorphic Steiner loops, and consequently, two isomorphic Steiner triple systems. However, it is important to note that the converse is not always true. But, by definition, if $\mathcal{L_N}$ coincides with the whole center of the two isomorphic Steiner loops $\mathcal{L}_{\mathcal{S}_i}$ (i.e., there are no further Veblen points of $\mathcal{S}_i$), the two isomorphic loops are necessarily isomorphic extensions. In \Cref{rmk:furtherveblen} we settle whether a point $(P,x)$ not in $\mathcal{L_N}$ is a Veblen point.
	
	\begin{remark}\label{rmk:furtherveblen}
		Let the Steiner loop $\mathcal{L_S}$ be a Schreier extension of $\mathcal{L_N}$ by $\mathcal{L_Q}$ with factor system $f$. A point $(P,x)$ is a Veblen point of $\mathcal{S}$ not in $\mathcal{N}$ if and only if $P$ is a Veblen point of $\mathcal{Q}$ and
		\begin{equation} \label{eq:furtherveblen}
			f(P,Q)+f(PQ,R)=f(Q,R)+f(P,QR),
		\end{equation}
		for every $Q,R\in\mathcal{Q}$. The condition \eqref{eq:furtherveblen} reflects the centrality of the element $(P,x)$ in $\mathcal{L_S}$.
	\end{remark}

	Also, two equivalent extensions are isomorphic as well, but the converse is not always true. However, reducing up to equivalence is useful to characterize isomorphism of extensions easily, as we show in \Cref{propisomorphism}. For simplicity of notation, if $\beta$ is an automorphism of $\mathcal{L_Q}$ and $f\in  \mathrm{Ext_S}(\mathcal{L_N},\mathcal{L_Q})$, we denote by $f\beta$ the factor system defined by 
	\begin{equation*}
		(P,Q)\longmapsto f\left(\beta(P),\beta(Q)\right).
	\end{equation*}

	\begin{proposition}\label{propisomorphism} 
		Two factor systems $f_1,f_2\in \mathrm{Ext_S}(\mathcal{L_N},\mathcal{L_Q})$ are isomorphic if and only if  
		\begin{equation}
			\alpha f_1 \sim f_2\beta
		\end{equation}
		for suitable $\alpha\in\mathrm{Aut}(\mathcal{L_N})$ and $\beta\in\mathrm{Aut}(\mathcal{L_Q})$. Moreover, the isomorphism between the corresponding loops has, up to equivalence, the following form
		\begin{equation} \label{isom2form}
			(P,x)\longmapsto (\beta(P), \alpha (x)).
		\end{equation}	
		
	\end{proposition}	
	
	\begin{proof} 
		Let $\Phi\colon \mathcal{L}_{\mathcal{S}_1} \longrightarrow \mathcal{L}_{\mathcal{S}_2}$ be an isomorphism between the Steiner loops corresponding, respectively, to the factor systems $f_1$ and $f_2$, and suppose that $\Phi$ defines an isomorphism of extensions.
		Since $\Phi(\mathcal{L_N})=\mathcal{L_N}$ and of course $\Phi (\mathcal{L_Q})=\mathcal{L_Q}$, we obtain that 
		\begin{equation}
			\Phi(\bar\Omega,x)=(\bar\Omega,\alpha(x)), \text{ for every } x\in\mathcal{L_N},
		\end{equation}
		where $\alpha$ is an automorphism of $\mathcal{L_N}$, and
		\begin{equation*}
			\Phi (P,\Omega')=(\beta(P),\varphi(P)), \text{ for every } P\in\mathcal{L_Q},
		\end{equation*}
		where $\beta$ is a suitable automorphism of $\mathcal{L_Q}$ and  $\varphi$ is a function $\mathcal{L_Q}\to\mathcal{L_N}$ mapping $\bar \Omega \mapsto \Omega'$.
		Thus,
		\begin{align*}
			\Phi(P,x)&=\Phi\left( (P,\Omega')\circ ( \bar \Omega, x)\right)\\
			&=(\beta(P), \varphi (P))\circ (\bar \Omega, \alpha (x))\\
			&=(\beta(P), \alpha (x)+\varphi (P)).
		\end{align*}
		If we multiply two elements in $\mathcal{L}_{\mathcal{S}_1}$, 
		\begin{equation}\label{eq:multiplicationisomorph}
			(P,x)\circ(Q,y)=(PQ,x+y+f_1(P,Q)),
		\end{equation}
		the isomorphism $\Phi$ maps the left-hand side of \Cref{eq:multiplicationisomorph} into 
		\begin{align*}
			&\left(\beta(P), \alpha (x)+\varphi (P)\right)\circ \left(\beta(Q), \alpha (y)+\varphi (Q)\right)\\
			=&\left(\beta(PQ), \alpha (x+y)+\varphi (P)+\varphi (Q) + f_2\left(\beta(P),\beta(Q)\right)\right)
		\end{align*}
		and the the right-hand side into 
		\begin{equation*}
			\big(\beta(PQ),\alpha(x+y)+\alpha f_1(P,Q)+\varphi(PQ)\big).
		\end{equation*}
		Therefore, for every $P,Q\in\mathcal{L_Q}$, the following holds
		\begin{equation*}
			f_2\left(\beta(P),\beta(Q)\right)=\alpha f_1(P,Q)+\varphi(PQ)+\varphi (P)+\varphi (Q).
		\end{equation*}
		Hence,
		\begin{equation*}
			f_2\beta=\alpha f_1+ \delta^1\varphi,
		\end{equation*}
		that is, $f_2\beta\sim \alpha f_1$. 
		By \Cref{lemmaequivalence}, up to equivalence $\Phi$ has the form in \Cref{isom2form}.
		
		Conversely, if $f_2\beta\sim \alpha f_1$, the map in \eqref{isom2form} defines an isomorphism of extensions. 
	\end{proof}

	Now we give an example of two isomorphic but not equivalent Schreier extensions.
	
	\begin{example}
		Let $\mathcal{L_N}=\{\Omega', 1\}$  be the unique loop of cardinality $2$ and $\mathcal{L_Q}$ be the Steiner loop corresponding to the $\mathrm{STS}(9)$. We can represent $\mathcal Q$ as the affine plane $\mathrm{AG}(3,2)$ with points and lines given by \Cref{fig:sts9Q}. 
		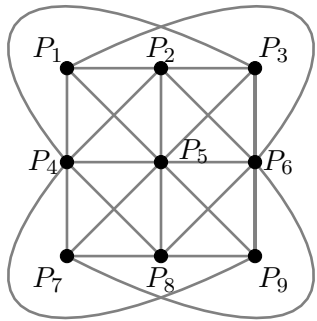
\begin{figure}[H]\centering
			\begin{tikzpicture}[line cap=round,line join=round,>=triangle 45,x=1cm,y=1cm,scale=1.25]
				\draw [line width=1pt,color=gray] (-1,-1)-- (1,1);
				\draw [line width=1pt,color=gray] (-1,-1)-- (1,-1);
				\draw [line width=1.5pt,color=gray] (1,-1)-- (1,1);
				\draw [line width=1pt,color=gray] (1,1)-- (-1,1);
				\draw [line width=1pt,color=gray] (-1,1)-- (-1,-1);
				\draw [line width=1pt,color=gray] (-1,0)-- (1,0);
				\draw [line width=1pt,color=gray] (0,1)-- (0,-1);
				\draw [line width=1pt,color=gray] (1,-1)-- (-1,1);
				\draw [line width=1pt,color=gray] (0,-1)-- (1,0);
				\draw [line width=1pt,color=gray] (1,0)-- (0,1);
				\draw [line width=1pt,color=gray] (0,1)-- (-1,0);
				\draw [line width=1pt,color=gray] (-1,0)-- (0,-1);
				\draw [shift={(-1.699214041911079,-0.8152542693528347)},line width=1pt,color=gray]  plot[domain=2.3117548637419008:4.2939481603980205,variable=\t]({-0.8196591288696354*3.96958988026553*cos(\t r)+0.5728515623271623*1.1900139979490725*sin(\t r)},{-0.5728515623271623*3.96958988026553*cos(\t r)+-0.8196591288696354*1.1900139979490725*sin(\t r)});
				\draw [shift={(1.699214041911079,-0.8152542693528345)},line width=1pt,color=gray]  plot[domain=2.3117548637419008:4.293948160398021,variable=\t]({0.8196591288696355*3.96958988026553*cos(\t r)+-0.5728515623271622*1.1900139979490725*sin(\t r)},{-0.5728515623271622*3.96958988026553*cos(\t r)+-0.8196591288696355*1.1900139979490725*sin(\t r)});
				\draw [shift={(-1.699214041911079,0.8152542693528347)},line width=1pt,color=gray]  plot[domain=2.3117548637419008:4.2939481603980205,variable=\t]({-0.8196591288696354*3.96958988026553*cos(\t r)+0.5728515623271623*1.1900139979490725*sin(\t r)},{0.5728515623271623*3.96958988026553*cos(\t r)+0.8196591288696354*1.1900139979490725*sin(\t r)});
				\draw [shift={(1.699214041911079,0.815254269352835)},line width=1pt,color=gray]  plot[domain=2.3117548637419008:4.2939481603980205,variable=\t]({0.8196591288696353*3.96958988026553*cos(\t r)+-0.5728515623271624*1.1900139979490725*sin(\t r)},{0.5728515623271624*3.96958988026553*cos(\t r)+0.8196591288696353*1.1900139979490725*sin(\t r)});
					\draw [fill=black] (-1,1) circle (2pt);
					\draw [fill=black] (0,1) circle (2pt);
					\draw [fill=black] (1,1) circle (2pt);
					\draw [fill=black] (-1,0) circle (2pt);
					\draw [fill=black] (0,0) circle (2pt);
					\draw [fill=black] (1,0) circle (2pt);
					\draw [fill=black] (-1,-1) circle (2pt);
					\draw [fill=black] (0,-1) circle (2pt);
					\draw [fill=black] (1,-1) circle (2pt);

					\node at (-1.20,1.2) {$P_1$};
					\node at (0,1.2) {$P_2$};
					\node at (1.20,1.2) {$P_3$};
					\node at (-1.25,0) {$P_4$};
					\node at (0.35,0.12) {$P_5$};
					\node at (1.25,0) {$P_6$};
					\node at (-1.20,-1.25) {$P_7$};
					\node at (0,-1.25) {$P_8$};
					\node at (1.20,-1.25) {$P_9$};
			\end{tikzpicture}\caption{$\mathrm{STS}(9)$ $\mathcal{Q}$}\label{fig:sts9Q}
		\end{figure}
		
		Consider the Schreier extension $\mathcal{L}_{\mathcal{S}_1}$ associated with the factor system $f_1$ such that
		\begin{equation*}
			f_1(P_3,P_6)=f_1(P_3,P_9)=f_1(P_6,P_9)=1
		\end{equation*}
		and $f_1$ is zero elsewhere.
		The automorphism $\beta $ of $\mathcal{L_Q}$ induced by the affine map $x \mapsto A x+b$ of $\mathrm{AG}(3,2)$, with 
		\begin{equation*}
			A=\left(\begin{array}{cc}
				1 & 1	\\
				0 & 1 \end{array} \right),
			\quad 
			b=\left(\begin{array}{c}
				-1 	\\
				0 \end{array} \right)
		\end{equation*}
		permutes the points of $\mathcal{Q}$ as
		\begin{equation*}
			\beta(P_i)=P_{\sigma(i)}, \quad \text{ with } \quad \sigma=(465)(789).
		\end{equation*}
		Consider the Steiner loop $\mathcal{L}_{\mathcal{S}_2}$ which is the Schreier extension associated with the factor system $f_2:=f _1\beta $, which is clearly defined by
		\begin{equation*}
			f_2(P_3,P_4)=f_2(P_3,P_8)=f_2(P_4,P_8)=1
		\end{equation*}
		and is zero elsewhere. By construction, $f_1$ and $f_2$ are isomorphic, but they are not equivalent. In fact, by \Cref{lemmaequivalence}, $f_1$ and $f_2$ are equivalent if and only if $f_1+f_2=\delta^1\varphi$, for a suitable function $\varphi$. If this were the case, then for each triple $\{P_i,P_j,P_k=P_iP_j\}$ of $\mathcal{Q}$ the following must hold
		\begin{equation}\label{eq:system}
			\varphi(P_i)+\varphi(P_j)+\varphi(P_k)=\delta^1(P_i,P_j)=(f_1+f_2)(P_i,P_j).
		\end{equation}
		Denoting $\varphi(P_i)$ with $X_i$, for every $i=1,\dots,9$, from \Cref{eq:system} we obtain the following linear system of twelve equations in the nine unknowns $X_i$ with scalars in the field $\mathrm{GF}(2)$.
		\begin{equation}\label{eq:linearsystem}
				{\footnotesize{\begin{pmatrix}
				1 & 1 & 1 & 0 & 0 & 0 & 0 & 0 & 0 \\
				1 & 0 & 0 & 1 & 0 & 0 & 1 & 0 & 0 \\
				1 & 0 & 0 & 0 & 1 & 0 & 0 & 0 & 1 \\
				1 & 0 & 0 & 0 & 0 & 1 & 0 & 1 & 0 \\
				0 & 1 & 0 & 1 & 0 & 0 & 0 & 0 & 1 \\
				0 & 1 & 0 & 0 & 1 & 0 & 0 & 1 & 0 \\
				0 & 1 & 0 & 0 & 0 & 1 & 1 & 0 & 0 \\
				0 & 0 & 1 & 1 & 0 & 0 & 0 & 1 & 0 \\
				0 & 0 & 1 & 0 & 1 & 0 & 1 & 0 & 0 \\
				0 & 0 & 1 & 0 & 0 & 1 & 0 & 0 & 1 \\
				0 & 0 & 0 & 1 & 1 & 1 & 0 & 0 & 0 \\
				0 & 0 & 0 & 0 & 0 & 0 & 1 & 1 & 1 \\
			\end{pmatrix}
			\begin{pmatrix}
				X_1  \\
				X_2 \\
				X_3 \\
				X_4 \\
				X_5 \\
				X_6 \\
				X_7 \\
				X_8 \\
				X_9 \\
			\end{pmatrix}=
			\begin{pmatrix}
				0  \\
				0 \\
				0  \\
				0 \\
				0 \\
				0  \\
				0  \\
				1  \\
				0  \\
				1  \\
				0  \\
				0  \\
			\end{pmatrix}}}
		\end{equation}
		Since this linear system can be proved to have no solution, such a function $\varphi$ cannot exist. 
	\end{example}

	By \Cref{propisomorphism}, non-equivalent but isomorphic factor systems $f_1$, $f_2$ are characterized by the relation $\alpha f_1 = f_2 \beta$, for suitable $\alpha\in\mathrm{Aut}(\mathcal{L_N})$ and $\beta\in\mathrm{Aut}(\mathcal{L_Q})$. This relation can be rewritten as
	\begin{equation}
		f_2 = \alpha f_1 \beta^{-1}.
	\end{equation}
	In this way, we can define a left action of the group $\mathrm{Aut}(\mathcal{L_N})\times\mathrm{Aut}(\mathcal{L_Q})$ on the set $\mathrm{Ext}_S(\mathcal{L_N},\mathcal{L_Q})/\mathrm{B}^2(\mathcal{L_Q},\mathcal{L_N})$ of non-equivalent extensions given by 
	\begin{equation*}
		(\alpha,\beta)(f)=\alpha f \beta^{-1},
	\end{equation*}
	whose orbits are the isomorphism classes of all the factor systems. 
	
	\begin{remark}
		The isomorphism classes of the factor systems in $\mathrm{Ext}_S(\mathcal{L_N},\mathcal{L_Q})$ are the orbits of the left action of the group $\mathrm{Aut}(\mathcal{L_N})\times\mathrm{Aut}(\mathcal{L_Q})$ on the set $\mathrm{Ext}_S(\mathcal{L_N},\mathcal{L_Q})/\mathrm{B}^2(\mathcal{L_Q},\mathcal{L_N})$ of non-equivalent extensions given by 
		\begin{equation*}
			(\alpha,\beta)(f)=\alpha f \beta^{-1}.
		\end{equation*}
		Indeed, by \Cref{propisomorphism}, non-equivalent but isomorphic factor systems $f_1$, $f_2$ are characterized by the relation $\alpha f_1 = f_2 \beta$, for suitable $\alpha\in\mathrm{Aut}(\mathcal{L_N})$ and $\beta\in\mathrm{Aut}(\mathcal{L_Q})$. 
	\end{remark}

	\section{Steiner operators}\label{extensions}
	
	In \Cref{section:schreierext} and \Cref{cohomology} we explored Schreier extensions of Steiner loops. In this section, we study a more general way of making extensions, that is, via a \emph{Steiner operator}. This concept allows us to construct, starting from an $\mathrm{STS}(u)$ and an $\mathrm{STS}(w)$, another Steiner triple system of order $(u+1)(w+1)-1$, with one as normal subsystem (which in this case is not central in general) and the other as the corresponding quotient. 
	
	We recall that a subloop of index two, that is, corresponding to a projective hyperplane, is always normal. However, by \Cref{maxnumbveblenpoints}, it is not central in non-projective Steiner triple systems. Also, a subloop of index four is generally not even normal, but again by \Cref{maxnumbveblenpoints}, if it is, it is not central in non-projective $\mathrm{STS}$s. However, both cases can be handled with Steiner operators. As we will see, in the case of a hyperplane, we only need to specify the diagonal blocks of a Steiner operator, thus, the problem of classifying $\mathrm{STS}(v)$s containing a projective hyperplane is reduced to classifying $\mathrm{STS}(\frac{v-1}{2})$s and symmetric Latin squares on $\frac{v-1}{2}$ letters with a fixed element in the main diagonal. In the case of index four, the diagonal blocks of the Steiner operator are not sufficient any longer to define the Steiner triple system, but we only need to fix one block other than the diagonal ones (see \Cref{thm:blocks}).

	\begin{definition}\label{SteinerOperator}
		Let $\mathcal{L_N}$ and $\mathcal{L_Q}$ be Steiner loops of order $n=u+1$ and $m=w+1$ with identity elements $\Omega'$ and $\bar{\Omega}$ respectively, and let $\mathrm{LS}(\mathcal{L_N})$ be the set of $n\times n$ Latin squares on $\mathcal{L_N}$. An operator $\Phi:\mathcal{L_Q}\times \mathcal{L_Q}\longrightarrow \mathrm{LS}(\mathcal{L_N})$, which maps the couple $(P,Q)$ into a Latin square
		$\Phi_{P,Q}: \mathcal{L_N}\times \mathcal{L_N}\longrightarrow \mathcal{L_N}$, is called a \emph{Steiner operator} if it fulfills the following conditions:
		\begin{enumerate}[(i)]
			\item the Latin square $\Phi_{\bar \Omega,\bar \Omega}$ is the (symmetric) multiplication table of $\mathcal{L_N}$;
			\item $\Phi_{Q,P}(y,x)=\Phi_{P,Q}(x,y)$, that is, $\Phi_{Q,P}$ is the transpose
			of $\Phi_{P,Q}$;
			\item $\Phi_{P,P}(x,x)=\Omega'$;
			\item$\Phi_{P,P  Q}(x, \Phi_{P,Q}(x,y))=y$;
		\end{enumerate}
		for all $(P,x),(Q,y)\in \mathcal{L_Q}\times \mathcal{L_N}$.
	\end{definition}
	
	\noindent Note that for $P=Q$ and $x=y$, conditions iii) and iv) lead to
	\begin{equation}\label{eq:rmkidentity}
		\Phi_{P,\bar \Omega}(x,\Omega')=x.
	\end{equation}
	
	\begin{theorem}\label{Extensions}
		Let $\mathcal{L_N}$ and $\mathcal{L_Q}$ be two Steiner loops of order $u+1$, $w+1$ and with identity elements $\Omega'$ and $\bar \Omega$, respectively. Let
		$\Phi:\mathcal{L_Q}\times \mathcal{L_Q}\longrightarrow \mathrm{LS}(\mathcal{L_N})$ be a Steiner operator.
		If we define on the set $\mathcal{L_Q}\times \mathcal{L_N}$ the following operation
		\begin{equation}
			(P, x) \circ (Q,y):=\big(PQ,\Phi_{P,Q}(x,y)\big),
		\end{equation}
		we obtain a Steiner loop $\mathcal{L_S}$  of order $v+1=(u+1)(w+1)$ with identity $\Omega=(\bar \Omega,\Omega')$. The subloop 
		$$\overline{\mathcal{L_N}}=\{(\bar \Omega, x)\mid x\in \mathcal{L_N}\}$$
		is a normal subloop of $\mathcal{L_S}$ isomorphic to $\mathcal{L_N}$, with corresponding quotient ${\mathcal L_S}/\overline{\mathcal{L_N}}$ isomorphic to $\mathcal{L_Q}$.
		Conversely, any Steiner loop with a proper normal subloop is isomorphic to the construction described above, for a suitable Steiner operator.
	\end{theorem}
	
	\begin{proof}
		
		Let $\mathcal{L_S}$ be defined by a Steiner operator $\Phi$ as above. If $(Q,y)$ and $(R, z)$ are two given elements in $\mathcal{L_S}$, then the equation
		$$(Q, y) \circ (P, x)=(R, z)$$
		has a unique solution $(P, x)$, where $P=QR $ and $x$ is the unique element in $\mathcal{L_N}$ such that
		$\Phi_{Q,QR}(y,x)=z$, that is, the column index of the element $z$ in row $y$ within the Latin square $\Phi_{Q,QR}$. 
		By \Cref{eq:rmkidentity}, the element $(\bar \Omega,\Omega')$ is the identity of $\mathcal{L_S}$. By condition ii) of \Cref{SteinerOperator}, the operation is commutative, by condition iii) $\mathcal{L_S}$ has exponent 2 and condition iv) is equivalent to $(P,x) \circ \big((P,x) \circ (Q, y))=(Q, y)$, that is, the totally symmetric property.	Thus, $\mathcal{L_S}$ is a Steiner loop. By Condition i), the subloop
		$$\overline{\mathcal{L_N}}=\{(\bar \Omega, x) \mid x\in \mathcal{L_N}\}$$
		is isomorphic to $\mathcal{L_N}$. For every $(Q,y),(R,z)\in\mathcal{L_S}$, $(\bar\Omega,n)\in\mathcal{L_N}$ the equation 
		\begin{equation}\label{eq:norm.st.op}
			\left(\left(Q,y\right) \left(R,z\right)\right)\circ \big(\bar\Omega,n\big) = \left(Q,y\right) \circ \left(\left(R,z\right) \left(P,x\right)\right),
		\end{equation}
		that is,
		\begin{equation*}
			\left(QR,\Phi_{Q,R}\left(y,z\right)\right)\circ \big(\bar\Omega,n\big) = \left(Q,y\right) \circ \left(RP,\Phi_{R,P}\left(z,x\right)\right),
		\end{equation*}
		is equivalent to 
		\begin{equation}\label{eq:norm.st.op2}
			\left(QR,\Phi_{QR,\bar\Omega}\left(\Phi_{Q,R}\left(y,z\right),n\right)\right)= \left(Q\cdot RP,\Phi_{Q,RP}\left(y,\Phi_{R,P}\left(z,x\right)\right)\right).
		\end{equation}
		\Cref{eq:norm.st.op2} implies $P=\bar\Omega$, that is, the solution $(P,x)$ of \Cref{eq:norm.st.op} belongs to $\mathcal{L_N}$ which, as a consequence, is normal. This argument proves the first part of the assertion.

		Conversely, consider a Steiner loop $\mathcal{L_S}$ with a normal subloop $\mathcal{L_N}$ and corresponding quotient loop $\mathcal{L_Q}$. Let $\pi\colon \mathcal{L_S}\longrightarrow\mathcal{L_Q}$ be the canonical epimorphism and $\sigma\colon\mathcal{L_Q}\longrightarrow \mathcal{L_S}$ a section with $\sigma(\mathcal{L_N})=\Omega$ and $\pi\sigma=\mathrm{id}_\mathcal{L_Q}$. Since for every $\pi(X)\in\mathcal{L_Q}$ it holds $\pi(X)=\pi(\sigma(\pi(X)))$, we have that 
		\begin{equation}\label{couplerepresentation}
			X=\sigma(\pi(X))\cdot x,
		\end{equation}
		with $x\in\mathcal{L_N}$.
		Since $\mathcal{L_N}$ is normal and by using the fact that  $\sigma(\pi(X))\sigma(\pi(Y))$ and $\sigma(\pi(X)\pi(Y))$ are in the same coset, we obtain that
		\begin{equation*}
			XY=\left(\sigma(\pi(X))\cdot x\right)\left(\sigma(\pi(Y))\cdot y\right)= \left(\sigma(\pi(X)\pi(Y))\right)\cdot \Phi_{\pi(X),\pi(Y)}(x,y )
		\end{equation*}
		for a suitable element $\Phi_{\pi(X),\pi(Y)}(x,y)$ of $\mathcal{L_N}$ depending on $\pi(X)$, $\pi(Y)$, $x$ and $y$. Since $\mathcal{L_S}$ is a loop, for any $\pi(X),\pi(Y)\in\mathcal{L_Q}$, $\Phi_{\pi(X),\pi(Y)}(-,-)$ defines a Latin square on $\mathcal{L_N}$. Thus, we can define an operator $\Phi\colon \mathcal{L_Q}\times\mathcal{L_Q}\longrightarrow\mathrm{LS}(\mathcal{L_N})$ such that $\Phi\colon (\pi(X),\pi(Y))\longmapsto\Phi_{\pi(X),\pi(Y)}$. Up to renaming the elements of $\mathcal{L_Q}$, every $X\in\mathcal{L_S}$ can be represented by a couple $(P,x)$ defined as in \Cref{couplerepresentation}, where $P=\pi(X)$. With this representation, the operation of $\mathcal{L_S}$ is the following:
		\begin{equation*}
			( P, x) \circ ( Q,y)=\left(PQ, \Phi_{P,Q}(x,y)\right).
		\end{equation*}
		Condition i) of \Cref{SteinerOperator} is trivially fulfilled since $x=(\bar\Omega,x)$ for every $x\in\mathcal{L_N}$. Condition ii) holds for commutativity, condition iii) comes from the exponent 2, and condition iv) reflects the totally symmetric property. 
\end{proof}
	
	In this case $\mathcal{L_S}$ is called an extension of $\mathcal{L_N}$ by $\mathcal{L_Q}$, and we say that the Steiner triple system $\mathcal{S}$ is an extension of $\mathcal{N}$ by $\mathcal{Q}$ as well.

	Roughly speaking, a Steiner operator $\Phi$ replaces the entry $PQ$ in the multiplication table of $\mathcal{L_Q}$ with the Latin square $\Phi_{P,Q}$.
		$$
		\;\begin{array}{c|ccc}
			& \dots  & Q & \dots  \\
			\hline
			\vdots  &   & \vdots  &    \\
			P& \cdots  & PQ & \cdots \\
			\vdots	&   &  \vdots &  
		\end{array}\;
		\quad \rightsquigarrow \quad
		\;\begin{array}{c|ccc}
			& \cdots  & \overbrace{}^{Q}  & \cdots  \\
			\hline
			\vdots  &   & \vdots  &    \\
			P \Big\{ & \dots  & \boxed{\Phi_{P,Q}}  & \dots \\
			\vdots	&   & \vdots  &  
		\end{array}\;$$
		In this way, we obtain the multiplication table of $\mathcal{L_S}$ by gluing together all the tables $\Phi_{P, Q}$ and recalling that in the first component of $\mathcal{L_S}=\mathcal{L_Q}\times\mathcal{L_N}$ we simply apply the multiplication of $\mathcal{L_Q}$.

	\begin{theorem}\label{thm:blocks}
		Consider an extension 
			\begin{equation*}
				\Omega' \to \mathcal{L_N} \to \mathcal{L_S} \to \mathcal{L_Q}  \to \bar\Omega,
			\end{equation*}
			of Steiner loops with $\mathcal{N}$ and $\mathcal{Q}$ of order $u$ and $w$ respectively. The $(u+1)(w+1) \times (u+1)(w+1)$ multiplication table of the Steiner loop $\mathcal{L_S}$ is completely determined by its $w+1$ diagonal symmetric blocks of size $(u+1) \times (u+1)$, and additional $\frac{w(w-1)}{6}$ blocks.
	\end{theorem}
	\begin{proof}
			Since the multiplication of $\mathcal{L_S}$ is given by
			\begin{equation*}
				(P,x) \circ ( Q,y)=\left(PQ, \Phi_{P,Q}(x,y)\right),
			\end{equation*}
			for a suitable Steiner operator $\Phi$, the multiplication table of $\mathcal{L_S}$ is described by the $(w+1)^2$ tables of size $(u+1)\times(u+1)$ corresponding to the Latin squares $\Phi_{P,Q}$, with $P,Q\in\mathcal{L_Q}$. Every Latin square $\Phi_{P,P}$ in the main diagonal uniquely determines the Latin squares $\Phi_{\bar \Omega, P}$ and $\Phi_{P, \bar \Omega}$. If $\{P,Q,R\}$ is a triple of $\mathcal{Q}$, then $\Phi_{P,Q}$ uniquely determines $\Phi_{P, R}$, $\Phi_{Q, R}$, and consequently $\Phi_{Q, P}$, $\Phi_{R, P}$, $\Phi_{R, Q}$ (see \Cref{tab:blocks} for a visual representation). Hence, once the blocks on the main diagonal are fixed, the remaining $w(w-1)$ blocks can be determined by specifying just $\frac{1}{6}$ of them.
			\begin{table}[H]
				\centering
				\begin{tabular}{c|c|c|c|c|c|c|c|c|}
					& $\bar \Omega$ & \dots & $ P$ & \dots & $Q$ &\dots & $R$ &\dots \\
					\hline
					$\bar\Omega$ &$\Phi_{\bar \Omega, \bar \Omega}$ &  & $\Phi_{\bar \Omega, P}$ &  & $\Phi_{\bar \Omega, Q}$ & & $\Phi_{\bar \Omega, R}$ &\dots \\
					\hline
					\vdots & & $\ddots$ & & & & & & \\
					\hline
					$P$ & $\Phi_{P, \bar \Omega}$ & & $\Phi_{P, P}$ & & \textcolor{blue}{$\Phi_{P, Q}$} & & \textcolor{blue}{$\Phi_{P, R}$} & \\
					\hline
					\vdots & & & & $\ddots$ & & & & \\
					\hline
					$Q$ & $\Phi_{Q, \bar \Omega}$ & & \textcolor{blue}{$\Phi_{Q, P}$} & & $\Phi_{Q, Q}$ & & \textcolor{blue}{$\Phi_{Q, R}$} & \\
					\hline
					\vdots & & & & &  & $\ddots$ & & \\
					\hline
					$R$ & $\Phi_{R, \bar \Omega}$ & & \textcolor{blue}{$\Phi_{R, P}$} & & \textcolor{blue}{$\Phi_{R, Q}$} & & $\Phi_{R, R}$ & \\
					\hline
					\vdots & & & & & & &  & $\ddots$\\
					\hline
				\end{tabular}
				\caption{Multiplication table of $\mathcal{L_S}$}
				\label{tab:blocks}
			\end{table}
	\end{proof}

		\begin{example}\label{sts19}
				Here we construct an $\mathrm{STS}(19)$, denoted with $\mathcal{S}$, containing a projective hyperplane $\mathcal{N}$ which must be isomorphic to the unique $\mathrm{STS}(9)$. The corresponding Steiner loop $\mathcal{L_N}$ is a normal subloop $\mathcal{L_S}$ of index $2$ in the Steiner loop of order $20$. We denote the corresponding quotient loop of order $2$ by $\mathcal{L_Q}=\{\bar\Omega, \bar 1 \}$. For $\mathcal{L_N}$ we fix the multiplication given by \Cref{tab:multLN}.
				
				\begin{table}[H]
						\centering
						\begin{tabular}{c|cccccccccc}
								& $\Omega'$ & $1$ & $2$ & $3$ & $4$ & $5$ & $6$ & $7$ & $8$ & $9$ \\
								\hline
								$\Omega'$ & $\Omega'$ & $1$ & $2$ & $3$ & $4$ & $5$ & $6$ & $7$ & $8$ & $9$ \\
								$1$ & $1$ & $\Omega'$ & $3$ & $2$ & $7$ & $9$ & $8$ & $4$ & $6$ & $5$ \\
								$2$ & $2$ & $3$ & $\Omega'$ & $1$ & $9$ & $8$ & $7$ & $6$ & $5$ & $4$ \\
								$3$ & $3$ & $2$ & $1$ & $\Omega'$ & $8$ & $7$ & $9$ & $5$ & $4$ & $6$ \\
								$4$ & $4$ & $7$ & $9$ & $8$ & $\Omega'$ & $6$ & $5$ & $1$ & $3$ & $2$ \\
								$5$ & $5$ & $9$ & $8$ & $7$ & $6$ & $\Omega'$ & $4$ & $3$ & $2$ & $1$ \\
								$6$ & $6$ & $8$ & $7$ & $9$ & $5$ & $4$ & $\Omega'$ & $2$ & $1$ & $3$ \\
								$7$ & $7$ & $4$ & $6$ & $5$ & $1$ & $3$ & $2$ & $\Omega'$ & $9$ & $8$ \\
								$8$ & $8$ & $6$ & $5$ & $4$ & $3$ & $2$ & $1$ & $9$ & $\Omega'$ & $7$ \\
								$9$ & $9$ & $5$ & $4$ & $6$ & $2$ & $1$ & $3$ & $8$ & $7$ & $\Omega'$
							\end{tabular}
						\caption{Multiplication table of $\mathcal{L_N}$}\label{tab:multLN}
					\end{table}
				The Latin square $\Phi_{\bar 1,\bar 1}$ is a symmetric table with the identity element $\Omega'$ occurring in the main diagonal. We choose, for instance, $\Phi_{\bar 1,\bar 1}$ to be given by \Cref{tab:phi11}.
				\begin{table}[H]
						\centering
						\begin{tabular}{c|cccccccccc}
								& $\Omega'$  & $1$  & $2$  & $3$ & $4$  & $5$  & $6$ & $7$ & $8$ & $9$  \\
								\hline
								$\Omega'$  & $\Omega'$  & $7$  & $6$  & $5$ & $4$  & $9$  & $8$ & $2$ & $1$ & $3$  \\
								$1$  & $7$  & $\Omega'$  & $5$  & $6$  & $2$  & $8$ & $9$ & $4$ & $3$ & $1$   \\
								$2$  & $6$  & $5$  &  $\Omega'$ & $7$  & $8$ & $2$  & $1$  & $3$ & $4$ & $9$  \\
								$3$	& $5$  & $6$  & $7$  & $\Omega'$ & $1$  & $3$  & $4$  & $9$ & $8$ & $2$  \\
								$4$	& $4$ & $2$  & $8$  & $1$ & $\Omega'$ & $5$  & $3$ & $7$ & $9$ & $6$ \\
								$5$ & $9$ & $8$ & $2$ & $3$ & $5$ &  $\Omega'$ & $7$  & $1$  & $6$  & $4$  \\
								$6$	& $8$  & $9$ & $1$  & $4$  & $3$  & $7$ & $\Omega'$ & $6$ & $2$  & $5$ \\
								$7$ & $2$ & $4$ & $3$ & $9$ & $7$ & $1$ & $6$ &  $\Omega'$ & $5$ & $8$ \\
								$8$ & $1$ & $3$ & $4$ & $8$ & $9$ & $6$ & $2$ & $5$ & $\Omega'$ & $7$ \\
								$9$ & $3$ & $1$ & $9$ & $2$ & $6$ & $4$ & $5$ & $8$ & $7$ & $\Omega'$
							\end{tabular}
						\caption{$\Phi_{\bar 1,\bar 1}$}\label{tab:phi11}
					\end{table}
				Each of the $45$ entries in the upper triangular part of $\Phi_{\bar 1,\bar 1}$ determines a triple of the $\mathrm{STS}(19)$. For instance, we can read from the table that
				\begin{equation}
						(\bar 1,4)\circ(\bar 1,1) = (\bar \Omega,2),
					\end{equation}
				meaning that $\{(\bar 1,4), (\bar 1,1), (\bar \Omega,2)\}$ is a triple of $\mathcal{S}$. Therefore, $\Phi_{\bar\Omega, \bar 1}(2,1)=4$ and $\Phi_{\bar\Omega, \bar 1}(2,4)=1$. In this way we find $45$ triples of $\mathcal{S}$, each of which gives two entries in the Latin square $\Phi_{\bar \Omega,\bar 1}$. Hence, the Latin square $\Phi_{\bar \Omega,\bar 1}$ is thoroughly determined as \Cref{tab:phiom1}.
				\begin{table}[H]
						\centering
						\begin{tabular}{c|cccccccccc}
								& $\Omega'$  & $1$  & $2$  & $3$ & $4$  & $5$  & $6$ & $7$ & $8$ & $9$  \\
								\hline
								$\Omega'$  & $\Omega'$  & $1$  & $2$  & $3$ & $4$  & $5$  & $6$ & $7$ & $8$ & $9$   \\
								$1$  & $8$  & $9$  & $6$  & $4$  & $3$  & $7$ & $2$ & $5$ & $\Omega'$ & $1$   \\
								$2$  & $7$  & $4$  & $5$ & $9$  & $1$ & $2$  & $8$  & $\Omega'$ & $6$ & $3$  \\
								$3$	& $9$  & $8$  & $7$  & $5$ & $6$  & $3$  & $4$  & $2$ & $1$ & $\Omega'$  \\
								$4$	& $4$ & $7$  & $8$  & $6$ & $\Omega'$ & $9$  & $3$ & $1$ & $2$ & $5$ \\
								$5$ & $3$ & $2$ & $1$ & $\Omega'$ & $5$ &  $4$ & $9$  & $8$  & $7$  & $6$  \\
								$6$	& $2$  & $3$ & $\Omega'$  & $1$  & $9$  & $8$ & $7$ & $6$ & $5$  & $4$ \\
								$7$ & $1$ & $\Omega'$ & $3$ & $2$ & $7$ & $6$ & $5$ &  $4$ & $9$ & $8$ \\
								$8$ & $6$ & $5$ & $4$ & $8$ & $2$ & $1$ & $\Omega'$ & $9$ & $3$ & $7$ \\
								$9$ & $5$ & $6$ & $9$ & $7$ & $8$ & $\Omega'$ & $1$ & $3$ & $4$ & $2$
							\end{tabular}
						\caption{$\Phi_{\bar \Omega,\bar 1}$}\label{tab:phiom1}
					\end{table}
				Note that the entries in $\Phi_{\bar \Omega,\bar \Omega}$, which is the multiplication table of $\mathcal{L_N}$, yield the $12$ triples of the hyperplane $\mathcal{N}$, thus we have all of the $57$ triples of the $\mathrm{STS}(19)$. 
				
				The elements of $\mathcal{L_S}$ are represented by couples $(P, x)$ in $\mathcal{L_Q}\times\mathcal{L_N}$ and the multiplication table of
				$\mathcal{L_S}$ is given by the four $10\times 10$ block matrices $\Phi_{P,Q}$.
				\begin{table}[H]
						\centering
						$\mathcal{L_S}:$
						\begin{tabular}{c|c}
								$\Phi_{\bar \Omega,\bar \Omega}$ & $\Phi_{\bar 1,\bar \Omega}$ \\
								\hline
								$\Phi_{\bar \Omega,\bar 1}$ & $\Phi_{\bar 1,\bar 1}$
							\end{tabular}
					\end{table}
				
			\end{example}
	
	\medskip
	
	In the same way as for Schreier extensions (see \Cref{def:equivalent} and \Cref{def:isomorphic}), we can define the concepts of \emph{isomorphic} and \emph{equivalent} extensions in the general case. 

	If two extensions are equivalent (respectively, isomorphic), we say that the corresponding Steiner operators are equivalent (respectively, isomorphic) as well.  
	
	While the concept of equivalent extensions of Steiner triple systems may seem less intuitive than that of isomorphic ones, the following theorem proves their natural emergence when applying a specific class of \emph{isotopies} to the Latin squares associated with Steiner operators. Indeed, Latin squares are studied up to isotopy, and this highlights, in our opinion, the relevance of equivalent extensions in this context.
	
	We recall that an \emph{isotopy} between two Latin squares $L_1$ and $L_2$ of the same size is a triple of bijections $(\gamma_1,\gamma_2,\gamma_3)$ such that $\gamma_3(L_1(i,j))=L_2(\gamma_1(i),\gamma_2(j))$. This means that $L_2$ is obtained by $L_1$ with a permutation of rows $\gamma_1$, a permutation of columns $\gamma_2$, and a permutation of symbols $\gamma_3$. 
	
	We have the following result. 
	
	\begin{theorem}
		Two Steiner operators $\Phi,\Psi \colon \mathcal{L_Q}\times \mathcal{L_Q} \to \mathrm{LS}(\mathcal{L_N})$ are equivalent if and only if there exists a map
		\begin{align*}
			\gamma\colon \mathcal{L_Q} &\longrightarrow \mathrm{Sym}(\mathcal{L_N})\\
			P &\longmapsto \gamma_P,
		\end{align*}
		with $\gamma_\Omega=\mathrm{id}$, such that, for every $P,Q\in\mathcal{L_Q}$, the triple of bijections $(\gamma_P,\gamma_Q,\gamma_{PQ})$ is an isotopy between the Latin squares $\Phi_{P,Q}$ and $\Psi_{P,Q}$.
	\end{theorem}
	
	\begin{proof}
		Let $\mathcal{L}_{\mathcal{S}_1} $ and $\mathcal{L}_{\mathcal{S}_2}$ be the two equivalent extensions corresponding to $\Phi$ and $\Psi$, respectively, and $\varphi\colon \mathcal{L}_{\mathcal{S}_1} \to \mathcal{L}_{\mathcal{S}_2}$ an isomorphism realizing the equivalence. Since $\mathcal{L_Q}$ and $\mathcal{L_N}$ are fixed pointwise by $\varphi$, we have that, for every $P\in \mathcal{L_Q}$
		\begin{equation*}
			\varphi (P,x)=(P,\gamma_P(x)),
		\end{equation*}
		where $\gamma_P\in\mathrm{Sym}(\mathcal{L_N})$ and $\gamma_\Omega=\mathrm{id}$.
		On the one hand, 
		\begin{equation*}
			\varphi\left((P,x)  (Q,y)\right)=\varphi \left(\left(PQ, \Phi_{P,Q}(x,y)\right)\right)=\left(PQ, \gamma_{PQ}\left(\Phi_{P,Q}(x,y)\right)\right).
		\end{equation*}
		On the other hand, 
		\begin{equation*}
			\varphi\left((P,x)  (Q,y)\right)=	\varphi\left(P,x\right)  \varphi\left(Q,y\right)=(P,\gamma_P(x))  (Q,\gamma_Q(y))=
			\left(PQ, \Psi_{P,Q}(\gamma_P(x),\gamma_Q(y))\right).
		\end{equation*}
		Hence, 
		\begin{equation*}
			\gamma_{PQ}\left(\Phi_{P,Q}(x,y)\right)= \Psi_{P,Q}(\gamma_P(x),\gamma_Q(y)),
		\end{equation*}
		that is, for every $P,Q\in\mathcal{L_Q}$, the triple $(\gamma_P,\gamma_Q,\gamma_{PQ})$ is an isotopy between the Latin squares $\Phi_{P,Q}$ and $\Psi_{P,Q}$.
		
		On the contrary, suppose there exists a map
		$\gamma\colon \mathcal{L_Q} \longrightarrow \mathrm{Sym}(\mathcal{L_N})$, $P \longmapsto \gamma_P$, with $\gamma_\Omega=\mathrm{id}$, such that, for every $P,Q\in\mathcal{L_Q}$, the triple $(\gamma_P,\gamma_Q,\gamma_{PQ})$ is an isotopy between the Latin squares $\Phi_{P,Q}$ and $\Psi_{P,Q}$. Then the map
	\begin{align*}
		\varphi \colon \mathcal{L}_{\mathcal{S}_1} &\longrightarrow \mathcal{L}_{\mathcal{S}_2}\\
		(P,x) &\longmapsto (P,\gamma_P(x)),
	\end{align*}
		is an isomorphism inducing the identity homomorphism on both $\mathcal{L_Q}$ and $\mathcal{L_N}$. Indeed,
		\begin{align*}
		\varphi\left((P,x)  (Q,y)\right)&=\varphi \left(\left(PQ, \Phi_{P,Q}(x,y)\right)\right)=\left(PQ, \gamma_{PQ}\left(\Phi_{P,Q}(x,y)\right)\right)\\
		&=\left(PQ, \Psi_{P,Q}(\gamma_P(x),\gamma_Q(y))\right)\\
		&=\varphi\left(P,x\right)  \varphi\left(Q,y\right).
		\end{align*}

	\end{proof}

\bigskip

\noindent
\textbf{Data availability}\\
We do not analyse or generate any datasets, because our work proceeds within a theoretical and mathematical approach. One can obtain the relevant materials from the references below.

\bigskip 

\noindent
\textbf{Conflict of interest statement}\\
All authors have no conflicts of interest.

\printbibliography
\end{document}